\documentclass[twoside,11pt]{article}
\usepackage[top=1in,bottom=1in,left=1in,right=1in]{geometry}
\usepackage{mkolar_definitions,subfigure,graphicx,color,bbm}
\usepackage[sort&compress]{natbib} 
\usepackage{algorithm,algpseudocode}
\usepackage{hyperref}
\usepackage{url}

\def\beq{\begin{equation}} 
\def\eeq{\end{equation}}
\def\beqn{\begin{eqnarray*}}
\def\eeqn{\end{eqnarray*}}
\def\Bitem{\begin{itemize}\setlength{\itemsep}{.2in}}
\def\bitem{\begin{itemize}\setlength{\itemsep}{.05in}}
\def\eitem{\end{itemize}}
\def\Benum{\begin{enumerate}\setlength{\itemsep}{.2in}}
\def\benum{\begin{enumerate}\setlength{\itemsep}{.05in}}
\def\eenum{\end{enumerate}}
\def\bmult{\begin{multline*}}
\def\emult{\end{multline*}}
\def\bcenter{\begin{center}}
\def\ecenter{\end{center}}
\def\bframe{\begin{frame}}
\def\eframe{\end{frame}}

\newcommand{\thmref}[1]{Theorem~\ref{thm:#1}}

\newcommand{\lemref}[1]{Lemma~\ref{lem:#1}}
\newcommand{\secref}[1]{Section~\ref{sec:#1}}

\renewcommand{\algref}[1]{Algorithm~\ref{alg:#1}}

\def\eps{\epsilon}

\def\comp{\mathsf{c}}

\def\smallsum{\textstyle\sum}
\newcommand{\tprod}{{\textstyle\prod}}
\newcommand{\tsum}{{\textstyle\sum}}
\newcommand{\tsup}{{\textstyle\sup}}

\renewcommand{\algorithmicrequire}{\textbf{Input:}}
\renewcommand{\algorithmicensure}{\textbf{Output:}}

\begin{document}

\title{Exact Asymptotics for the Scan Statistic and Fast Alternatives}
\author{James Sharpnack and Ery Arias-Castro}
\date{University of California, San Diego}
\maketitle

\begin{abstract}
We consider the problem of detecting a rectangle of activation in a grid of sensors in d-dimensions with noisy measurements.  
This has applications to massive surveillance projects and anomaly detection in large datasets in which one detects anomalously high measurements over rectangular regions, or more generally, blobs.  
Recently, the asymptotic distribution of a multiscale scan statistic was established in \citep{kabluchko2011extremes} under the null hypothesis, using non-constant boundary crossing probabilities for locally-stationary Gaussian random fields derived in \citep{chan2006maxima}.
Using a similar approach, we derive the exact asymptotic level and power of four variants of the scan statistic: an oracle scan that knows the dimensions of the activation rectangle; the multiscale scan statistic just mentioned; an adaptive variant; and an $\eps$-net approximation to the latter, in the spirit of \citep{MGD}.   
This approximate scan runs in time near-linear in the size of the grid and achieves the same asymptotic power as the adaptive scan.
We complement our theory with some numerical experiments.

\end{abstract}

\section{Introduction}
\label{sec:intro}

Detecting anomalies in networks is important in a number of areas, such as sensor arrays \citep{radiation,overview-sensor}, digital images (incl.~satellite, medical, etc.) \citep{artificial-natural,fire-detection,medical-survey,breast-tumor,brain-tumor}, syndromic surveillance systems \citep{rotz2004advances,heffernan2004ssp,wagner2001esv}, and many more.
The scan statistic \citep{Kul} is by far the most popular approach, and is given different names in engineering, such as the method of matched filters or deformable templates \citep{medical-survey}. 
It was perhaps first introduced for finding patterns in point clouds \citep{Naus65,glaz01} and is now applied to any setting where the goal is to detect a ``localized" anomaly. 
In statistics, it corresponds to the generalized likelihood ratio test after a particular model is assumed, and as such is even more widely applicable, being the most omnibus approach to hypothesis testing.

Focusing on the detection of anomalies in networks, which includes spatiotemporal data, first order theoretical performance bounds are established in a small number of papers, such as \citep{MR2604703,MGD,morel,cluster}.  More refined results establishing weak convergence are even fewer.  \cite{jiang} considers the scan over rectangles in a grid of independent random variables with negative expectation, while \cite{boutsikas} study the scan over intervals of given length in a Bernoulli sequence.  Both works are rely heavily on the Chen-Stein Poisson approximation.
Still in the context of the one-dimensional lattice, but now with standard normal random variables, \cite{sieg95} provide 
a weak convergence for the normalized scan over all intervals.  Concretely, suppose that $y(1),\dots,y(n)$ are iid standard normal, and define
\[
Z_n = \max_{1 \le i_1 \le i_2 \le n} \frac1{\sqrt{i_2-i_1+1}} \sum_{i=i_1}^{i_2} y(i).
\]
Then \cite{sieg95} show that, for all $\tau \in \RR$, 
\[
\lim_{n \to \infty} \PP\big(Z_n \ge u_n(\tau)\big) = 1 - e^{-e^{-\tau}}, \quad u_n(\tau) := \sqrt{2 \log n} + \frac{\frac12 \log(2 \log n) + \kappa + \tau}{ \sqrt{2 \log n}},
\]
for some numeric constant $\kappa$.
This was recently extended to higher dimensions, for scans over hypercubes and hyperrectangles, by \cite{kabluchko2011extremes}.
Formally, define $[n] = \{1, \dots, n\}$ and assume that $(y(\ib) : \ib \in [n]^d)$ are iid standard normal.  A (discrete) hyperrectangle is of the form $[a_1, b_1] \times \cdots \times [a_d, b_d] \subset [n]^d$.
Let $\Rcal$ denote the class of all discrete hyperrectangles of $[n]^d$ and define the scan over $\Rcal$ as
\beq \label{scan}
Z_n = \max_{R \in \Rcal} \frac1{\sqrt{|R|}} \sum_{\ib \in R} y(\ib),
\eeq
where $|R|$ denotes the number of nodes in $R$, equal to $\prod_j (b_j -a_j +1)$ when $R = \times_j [a_j, b_j]$.
\cite{kabluchko2011extremes} shows that, for all $\tau \in \RR$, 
\beq \label{kabluchko}
\lim_{n \to \infty} \PP\big(Z_n \ge u_n(\tau)\big) = 1 - e^{-e^{-\tau}}, \quad u_n(\tau) := \sqrt{2 d \log n} + \frac{(d-\frac12) \log(2 d \log n) + \kappa + \tau}{ \sqrt{2 d\log n}},
\eeq
for some constant $\kappa$ depending only on the dimension $d$.
These results allow, in theory, to control the (asymptotic) level the test based on the scan statistic if, indeed, the data is iid standard normal when no anomaly is present and an anomaly comes in the form of a rectangle with elevated mean.  This is what we assume throughout the paper. 

\begin{quote}
{\em Contribution 1.}  We establish a weak convergence result when an anomaly is present  which, in theory, allows for precise (asymptotic) power calculations.
\end{quote}

Besides the scan statistic \eqref{scan}, we study other variants.  One of them, already considered in \citep{cluster,MR2604703}, is based on a finer normalization for the scans at different scales.  In detail, define the class of rectangles with shape $\hb \in [n]^d$ as 
\beq \label{Rh}
\Rcal(\hb) = \left\{ \times_{j = 1}^d [v_j, v_j + h_j] : v_j \in [n-h_j], \forall j \in [d] \right\},
\eeq
and let $Z_{n,\hb}$ denote the scan over $\Rcal(\hb)$, defined as in \eqref{scan} but with $\Rcal(\hb)$ in place of $\Rcal$.
We consider the test that rejects if there is $\hb$ such that $Z_{n,\hb} \ge u_{n,\hb}(\tau)$, for some explicit critical values $u_{n,\hb}(\tau)$ defined later. 
We refer to this procedure as the (scale or shape) adaptive scan.
We note that in the first order analyses found in \citep{cluster,MR2604703}, $u_{n,\hb}(\tau)$ only depends on $\|\hb\|_1 := h_1 + \dots + h_d$, which is not quite true in our situation.

\begin{quote}
{\em Contribution 2.}  We establish weak convergence results for the adaptive scan, both when an anomaly is absent and when it is present. 
\end{quote}

Both the scan and the adaptive scan are computationally intensive.  With proper implementation, they can be computed in $O(n^{2d})$ basic operations, which may nevertheless be prohibitive for scans over large networks.  
For example, a typical 2D digital image is  of size $n \times n$, where $n$ is in the order of $10^3$, resulting a computational complexity on of order $10^{12}$ basic operations.
Aware of that, \cite{MGD,cluster} and \cite{MR2604703} propose to approximate the scan statistic by scanning over a subset of rectangles that is sufficiently dense in $\Rcal$.
For a given metric $\delta$ over $\Rcal$, we say that $\Rcal_\eps$ is an $\eps$-covering if, for all $R \in \Rcal$, there is $R' \in \Rcal_\eps$ such that $\delta(R, R') \le \eps$.  
Both \cite{MGD,cluster} and \cite{MR2604703} construct different $\eps$-coverings which can be scanned in roughly $O(n^d)$ basic operations, and show that, when $\eps = \eps_n \to 0$ sufficiently slowly, scanning over an $\eps$-covering yields the same first-order asymptotic performance.  

\begin{quote}
{\em Contribution 3.}  
We establish weak convergence results for the adaptive scan over a given $\eps$-covering, both when an anomaly is absent and when it is present. 
We also construct a new $\eps$-covering and design an efficient way to scan over it using on the order of $O(n^d)$ basic operations when $\eps$ is not too small.
\end{quote}

As a benchmark we consider an oracle which knows the shape $\hb^\star$ of the anomalous rectangle (but is ignorant of its location) and therefore only scans over rectangles with the same shape, meaning, over $\Rcal(\hb^\star)$.

\begin{quote}
{\em Contribution 4.}  
We establish weak convergence results for the oracle scan, both when an anomaly is absent and when it is present. 
\end{quote}

We note that our method of proof is largely borrowed from \cite{kabluchko2011extremes}, whose approach is based on extensive work of \cite{chan2006maxima} on the extrema of Gaussian random fields, and related topics, and the Chen-Stein Poisson approximation 
\citep{arratia1989two}.

We complement our theoretical findings with some numerical experiments that we performed to compare these various methods, meaning, the oracle scan, the scan, the adaptive scan, and the adaptive scan over an $\eps$-covering.

The rest of the paper is organized as follows.
In \secref{main}, we set the framework and state the theoretical results announced above, and in \secref{experiments}, we present the result of our numerical experiments.  
We briefly discuss some extensions and open problems in \secref{discussion}, while the technical proofs are gathered in \secref{proofs}.

Before we continue, we pause to introduce some notation.
We already used the notation $[n] = \{ 1,\ldots,n\}$ for any positive integer $n$.
Cartesian products of sets are denoted with the $\times$ operator and for a set $A$ and integer $k \ge 1$, $A^k = A \times \cdots \times A$, $k$ times.
All vectors are bolded and scalars are not.
Some special vectors are $\zero = \{0,\ldots,0\}$, $\one = \{ 1, \ldots, 1\}$, and the $j$th canonical basis vector $\eb_j = \{0,\ldots,0,1,0,\ldots,0\}$ with the $1$ in the $j$th component.
A vector $\ab$ in dimension $d$ will have components denoted $a_1, \dots, a_d$.
The rectangle with endpoints $\ab, \bb$ be denoted $[\ab,\bb] = \times_{i=1}^d [a_i,b_i]$.
The symbol $\circ$ indicates the component-wise product for vectors and matrices, division between vectors denoted $\ab/\bb$ is component-wise.
The Lebesgue measure in $\RR^d$ will be denoted by $\lambda$.
For a discrete set $R$, $|R|$ denotes its cardinality.
For a vector $\ab$, $\|\ab\|$ and $\|\ab\|_1$ denote its Euclidean and $\ell_1$ norms, respectively.
For a set $A$, $I\{A\}$ (sometimes $\mathbbm{1}_A$) will denote the indicator of $A$.
We use Bachmann-Landau notation to compare infinite sequences.  
For example, if $\{a_n\}_{n=1}^\infty, \{b_n \}_{n=1}^\infty$ are such that $a_n / b_n \rightarrow 0$ then $a_n = o(b_n)$ and $b_n = \omega(a_n)$.
For stochastic sequences, if the convergence is in probability then this is denoted by a subscript as in $a_n = o_\PP(b_n)$.
%

\section{Model, methodology, and theory}
\label{sec:main}

We assume that we are given one snapshot of measurements from a sensor array in $d$-dimensional space.  This array is arranged by placing one sensor at each grid point in $[n]^d$.  An important example is that of digital images, from CCD or CMOS cameras, or other modalities such as MRI.  It also encompasses video (incl.~fMRI), by letting one dimension represent time (in some unit), although the time dimension is often treated in a special way.

We denote the measurement at sensor $\ib \in [n]^d$ by $y(\ib)$ and model this as a signal vector with additive white Gaussian noise,
\beq \label{model}
y(\ib) = x(\ib) + \xi(\ib), \quad \ib \in [n]^d,
\eeq
where $x$ is the signal and $\xi$ is white standard normal noise, or in vector notation,
\[
\yb = \xb + \xib,
\]
where $\yb, \xb \in \RR^{n^d}$ and $\xib$ is a standard normal vector in $d$ dimensions.
We address the problem of deciding whether the signal $x$ is nonzero, formalized as the following hypothesis testing problem:
\begin{equation}
\label{eq:test}
\begin{aligned}
  &H_0: \xb = \zero, \\
  &H_1: \xb \in \Xcal_\mu,
\end{aligned}
\end{equation}
for some parameter $\mu$, that will be interpreted as the signal size, and some class $\Xcal_\mu \subset \RR^{n^d}$ parametrized by $\mu$ and with the property $\zero \notin \Xcal_\mu$.
While $H_0$ represents `business as usual', $H_1$ would indicate that there is some anomalous activity, here modeled by $x$.

We address the situation where the signal has substantial `energy' over a rectangle of unknown shape.
Given a signal $\yb = (y(\ib) : \ib \in [n]^d)$, define its Z-score over a subset $R \subset [n]^d$ as 
\beq \label{Z}
y[R] = \frac1{\sqrt{|R|}}{\sum_{\ib \in R} y(\ib)}.
\eeq
Recalling the class $\Rcal(\hb)$ of rectangles of with shape $\hb$ defined in \eqref{Rh}, let $\Xcal_\mu(\hb)$ denote the following set of signals:
\beq \label{alt}
\Xcal_\mu(\hb) = \Big\{\xb \in \RR^{n^d} : \min_{R \in \Rcal(\hb)} x[R] \ge \mu \Big\}.
\eeq
Rectangles are useful in practice because of their ease of interpretation and implementation, and also because they are building blocks for more complicated shapes.  They are also more amenable to a sharp asymptotic analysis, which is the focus in this paper.
See the discussion in \secref{discussion}.

Let $\hb^\star$ denote the shape of rectangle of activation, defined as the shape $\hb$ such that ${\rm supp}(\xb) \in \Rcal(\hb)$.
For the sake of clarity and ease of analysis, we assume that we are given integers $1 \le \underline h \le \overline h \le n/e$ (where $e = \exp(1)$) such that $\hb^\star \in [\underline h, \overline h]^d$.  
Redefine $\Rcal$ as 
\beq \label{Rcal}
\Rcal = \bigcup_{\hb \in [\underline h, \overline h]^d} \Rcal(\hb).
\eeq
We know that, under the alternative, the signal is elevated over a rectangle in $\Rcal$, namely
\[
\xb \in \Xcal_\mu := \bigcup_{\hb \in [\underline h, \overline h]^d} \Xcal_\mu(\hb).
\]
Our analysis is asymptotic with respect to the grid size diverging to infinity, $n \to \infty$.  While the grid dimension $d$ remains fixed, $\mu$, $\overline h$, and $\underline h$ are allowed to depend on $n$.  In fact, throughout this paper we assume that 
\beq \label{underline_h}
\underline h = \underline h_n \text{ satisfies } \underline h / \log n \rightarrow \infty, \text{ as } n \to \infty,
\eeq
to avoid special cases and complications that arise when including very small rectangles in the scan.  

As mentioned in the Introduction, we will use the oracle scan as a benchmark.  Instead of \eqref{eq:test}, the oracle, which knows the shape $\hb^\star$, is faced with the simpler alternative:
\beq \label{oracle}
H_1^\star: \xb \in \Xcal_\mu(\hb^\star).
\eeq

We take the asymptotic Neyman-Pearson approach in which we control the asymptotic probability of type I error (aka false rejection).
Consider a test $T(\yb)$ which evaluates to $1$ if it rejects $H_0$ and $0$ otherwise.
Throughout, we assume that a level $\alpha \in (0,1)$ is given and we control the tests at the exact asymptotic level $\alpha$, which means that
\[
\lim_{n \rightarrow \infty} \PP_0 \{ T(\yb) = 1 \} = \alpha,
\]
where $\PP_0$ indicates the distribution of $\yb$ under $H_0$.
The left-hand side is called the asymptotic size of the test $T$.
For all of the test statistics that we will study, we provide a threshold that gives us such a type I error control.
Once the size of the test is under control, we examine the power of the test.  We choose to focus on the minimum power over the set of alternatives,  which in the asymptote is defined as
\[
\beta(T) = \lim_{n \rightarrow \infty} \inf_{\xb \in \Xcal} \PP_\xb \{ T(\yb) = 1 \},
\]
where $\PP_\xb$ denotes the distribution of $\yb$ under model \eqref{model}.

\subsection{The Oracle scan}

When tasked with finding a rectangle of activation in a $d$-dimensional lattice, the problem is made easier if one knows the precise shape of the active rectangle.
Having access to an oracle that provides the shape of the anomalous region simplifies the alternative down to \eqref{oracle}.
In this situation, one would naturally restrict the scan to rectangles with shape $\hb^\star$.  
We called this procedure the {\em oracle scan} in the Introduction.
Given a critical value $u$, the oracle scan test is defined as
\beq \label{scan_o}
T_o(\yb) = I \Big\{ y[R] > u \textrm{ for some } R \in \Rcal(\hb^\star) \Big\}.
\eeq

\subsubsection{Asymptotic theory}
Define the following critical value 
\beq \label{u_o}
u_n(\tau) = v_n + \frac{(2d-1) \log(v_n) + \kappa + \tau}{v_n},
\eeq
where 
\beq \label{v_o}
v_n = \sqrt{2 \smallsum_j \log (n/h^\star_j)}, \quad \kappa = - \log ( \sqrt{2\pi} ).
\eeq
Given a level $\alpha \in (0,1)$, we choose 
\beq \label{tau}
\tau = \tau_\alpha = - \log (- \log (1 - \alpha) ).  
\eeq

\begin{theorem}
\label{thm:oracle_type1}
Suppose that $\min_i h_i^\star = \omega(\log n)$. 
The oracle scan test \eqref{scan_o} with critical value \eqref{u_o} and $\tau$ chosen as in \eqref{tau}, has the following asymptotic size 
\[ 
\lim_{n \rightarrow \infty} \PP_0 \{T_o(\yb) = 1\} = 1 - e^{-e^{-\tau}} = \alpha.
\]
\end{theorem}

Let $\bar\Phi$ denote the survival function of the standard normal distribution. 

\begin{theorem}
\label{thm:oracle_type2}
Suppose that $\min_j h_j^\star = \omega(\log n)$. 
The oracle scan test \eqref{scan_o} with critical value \eqref{u_o}-\eqref{tau} has the following asymptotic power 
\[
\lim_{n \rightarrow \infty} \inf_{\xb \in \Xcal(\hb^\star)} \PP_\xb \{ T_o(\yb) = 1 \} = 
\begin{cases}
1, &\mu - v_n \rightarrow \infty,\\
\alpha + (1 - \alpha) \bar\Phi(c), &\mu - v_n \rightarrow c, \textrm{ for some } c \in \RR,\\
\alpha, &\textrm{otherwise},
\end{cases}\]
where $v_n$ is defined in \eqref{v_o}.
\end{theorem}

\subsubsection{Computational complexity} \label{sec:comp_o}
While a naive implementation runs in $O(n^d \prod_j h_j^\star)$ time, the oracle scan can be computed in $O(n^d \log n)$ time using the Fast Fourier Transform (FFT), which is generally faster when the $h_j^\star$'s are not too small.
Specifically, let $b_\hb$ be the boxcar function with shape $\hb$, namely
\[
b_\hb(\ib) = \prod_{j = 1}^d I\{i_j \le h_j\},
\]
and let $\ast$ denote the convolution operator, so that, for $f : [n]^d \mapsto \RR$,
\[
(f \ast b_\hb) (\tb) = \sum_{\ib \in [n]^d} f(\tb + \ib) b_\hb(\ib) = \sum_{\ib \in [\hb]} f(\ib + \tb).
\]
Thus, computing the convolution $y \ast b_\hb$ amounts to computing $(y[R] : R \in \Rcal(\hb))$, and using the FFT, this convolution can be computed in $O(n^d \log n)$ time.
And the oracle scan test is based on the maximum of $y \ast b_{\hb^\star}$.

\subsection{The multiscale scan}

Perfect knowledge of the shape of the true rectangle of activation is rare. A simple solution to this problem is to scan over all rectangles in the class $\Rcal$ and report the largest observed $Z$-score.
Formally, given a critical value $u$, the {\em multiscale scan test} is  
\beq \label{scan_m}
T_m(\yb) = I \Big\{ y[R] > u \textrm{ for some } R \in \Rcal \Big\}.
\eeq
This is the test based on the scan statistic as defined in \eqref{scan}, except that $\Rcal$ is now defined as in \eqref{Rcal}.

\subsubsection{Asymptotic theory}
Define the following critical value 
\beq \label{u_m}
u_n(\tau) = v_n + \frac{(4d-1) \log(v_n) + \kappa + \tau}{v_n},
\eeq
where 
\beq \label{v_m}
v_n = \sqrt{2 d \log (n/\underline h)}, \quad \kappa = - \log (4^d \sqrt {2 \pi}).
\eeq

\citep[Th 1.2]{kabluchko2011extremes} establishes the asymptotic size of the multiscale scan test when $\underline h = 1$ and $\overline h = n$.  We do the same, when $\underline h = \omega(\log n)$ and $\overline h \le n/e$.  We note that, because of that, the critical value that we use \eqref{u_m} is different from the one that \cite{kabluchko2011extremes} uses \eqref{kabluchko}: the constants denoted by $\kappa$ in both places are in fact different, and the $(4d-1)$ factor in \eqref{u_m} is a $(2 d-1)$ factor in \eqref{kabluchko}.

\begin{theorem}  \citep{kabluchko2011extremes}
\label{thm:multiscale_type1}
Suppose that $\underline h = \omega(\log n)$.
The multiscale scan test \eqref{scan_m} with critical value \eqref{u_m}-\eqref{tau}, has the following asymptotic size 
\[ 
\lim_{n \rightarrow \infty} \PP_0 \{T_m(\yb) = 1\} = 1 - e^{-e^{-\tau}} = \alpha.
\]
\end{theorem}

\begin{theorem}
\label{thm:multiscale_type2}
Suppose that $\min_i h_i^\star = \omega(\log n)$. 
The multiscale scan test \eqref{scan_m} with critical value \eqref{u_m}-\eqref{tau} has the following asymptotic power 
\[
\lim_{n \rightarrow \infty} \inf_{\xb \in \Xcal(\hb^\star)} \PP_\xb \{ T_m(\yb) = 1 \} = 
\begin{cases}
1, &\mu - v_n \rightarrow \infty,\\
\alpha + (1 - \alpha) \bar\Phi(c), &\mu - v_n \rightarrow c, \textrm{ for some } c \in \RR,\\
\alpha, &\textrm{otherwise},
\end{cases}\]
where $v_n$ is defined in \eqref{v_m}.
\end{theorem}

Compared with the oracle scan test (see \thmref{oracle_type2}), the multiscale scan test (at the same level) has strictly less asymptotic power in general.  For example, suppose that $\underline h \asymp n^a$ and $h_j^\star \asymp n^b$ for all $j$, for some fixed $0 < a < b < 1$.  In that case, to have power tending to one, the oracle scan requires $\mu - \sqrt{1-b} \sqrt{2 d \log n} \to \infty$, while the multiscale scan requires $\mu - \sqrt{1-a} \sqrt{2 d \log n} \to \infty$.

\subsubsection{Computational complexity}
Using the FFT, the multiscale scan statistic can be computed in $
O\big((n^{2d}/\underline h^d) \log n \big)$ time, since each shape can be scanned in $O\big(n^{d} \log n \big)$ as we saw in \secref{comp_o}, and there are $O(n^d / \underline h^d)$ shapes in total in~$\Rcal$.

\subsection{The adaptive multiscale scan}

While the multiscale scan uses the same threshold $u_m$ for all rectangle sizes, it ignores the fact that detecting small rectangles (at the finer scales) is more difficult than detecting large rectangles.
The approach advocated in \citep{MR2604703,cluster} is a refinement of the multiscale scan in that a different threshold is used at each scale (i.e., rectangle size).
Formally, given (possibly) shape-dependent critical values $u_\hb$, the {\em adaptive multiscale scan test} is  
\beq \label{scan_a}
T_a(\yb) = I \Big\{ y[R] > u_\hb, \textrm{ for some } \hb \in [\underline h, \overline h]^d \text{ and } R \in \Rcal(\hb) \Big\}.
\eeq
If in fact $u$ does not depend on $\hb$, then this is the multiscale scan test \eqref{scan_m}.

\subsubsection{Asymptotic theory}
Define the following shape-dependent critical value 
\beq \label{u_a}
u_{n,\hb}(\tau) = v_{n,\hb} + \frac{(4d-1) \log(v_{n,\hb}) + \kappa + \tau}{v_{n,\hb}},
\eeq
where 
\beq \label{v_a}
v_{n,\hb} = \sqrt{2 \smallsum_j \log \big[\frac{n}{h_j} \big(1 + \log \frac{h_j}{\underline h}\big)^2\big]}, \quad \kappa = - \log (4^d \sqrt {2 \pi}).
\eeq

\begin{theorem}
\label{thm:adaptive_type1}
Suppose that $\underline h = \omega(\log n)$.
The adaptive multiscale scan test \eqref{scan_a} with critical value \eqref{u_a}-\eqref{tau}, has the following asymptotic size 
\[ 
\lim_{n \rightarrow \infty} \PP_0 \{T_a(\yb) = 1\} = 1 - e^{-e^{-\tau}} = \alpha.
\]
\end{theorem}

\begin{theorem}
\label{thm:adaptive_type2}
Suppose that $\min_j h_j^\star = \omega(\log n)$. 
The adaptive multiscale scan test \eqref{scan_m} with critical value \eqref{u_m}-\eqref{tau} has the following asymptotic power 
\[
\lim_{n \rightarrow \infty} \inf_{\xb \in \Xcal(\hb^\star)} \PP_\xb \{ T_a(\yb) = 1 \} = 
\begin{cases}
1, &\mu - v_{n,\hb^\star} \rightarrow \infty,\\
\alpha + (1 - \alpha) \bar\Phi(c), &\mu - v_{n,\hb^\star} \rightarrow c, \textrm{ for some } c \in \RR,\\
\alpha, &\textrm{otherwise},
\end{cases}\]
where $v_{n,\hb}$ is defined in \eqref{v_a}.
\end{theorem}

The adaptive multiscale scan test (at the same level) happens to achieve the same asymptotic power as the oracle scan (see \thmref{oracle_type2}) in the important case where $\hb^\star$ is not too large.  
Indeed, suppose for example that $\min_j h_j^\star = O(n^b)$ for some fixed $0 < b < 1$.  Letting $v_n^\star$ denote the $v_n$ in \eqref{v_o}, we obviously have $v_{n, \hb^\star} \ge v_n^\star$, and also 
\[
v_{n, \hb^\star} 
\le v_n^\star \sqrt{1 + (v_n^\star)^{-2} d \log \log n}
\le v_n^\star \big[1 + \tfrac12 (v_n^\star)^{-2} d \log \log n\big]
= v_n^\star + O\Big(\frac{\log\log n}{\sqrt{\log n}}\Big) 
= v_n^\star + o(1),
\]
so that $v_{n, \hb^\star} = v_n^\star + o(1)$.

\subsubsection{Computational complexity}
The computational cost for computing the adaptive multiscale scan is the same as that for computing the multiscale scan, i.e., $
O\big((n^{2d}/\underline h^d) \log n \big)$ time.

\subsection{Approximate adaptive multiscale scan}
\label{sec:eps-scan}

The computational complexity of the adaptive multiscale scan, which is quadratic in the grid size, may be prohibitive in some situations.  We provide now an algorithm that has nearly linear computation time while achieving the same asymptotic power.  
Inspired by the multiscale approximation developed in \citep{MGD,cluster,MR2604703}, we accomplish this by effectively scanning only over a subset of the rectangles that form an $\epsilon$-covering for $\Rcal$.
We recall that, given a metric $\delta$ over $\Rcal$, $\Rcal_\eps \subset \Rcal$ is an $\eps$-covering of $\Rcal$ for $\delta$ if, for all $R \in \Rcal$, there is $R' \in \Rcal_\eps$ such that $\delta(R, R') \le \eps$.  
Recall the definition of $\xi$ in \eqref{model}.
We use the canonical metric for the Gaussian random field $\{\xi[R], R \in \Rcal\}$, which is given by 
\beq \label{delta}
\delta^2(R_0, R_1) = \EE (\xi[R_0] - \xi[R_1])^2 = 2 \bigg( 1 - \frac{|R_0 \cap R_1|}{\sqrt{|R_0||R_1|}} \bigg), \quad \forall R_0, R_1 \in \Rcal.
\eeq
Given an $\epsilon$-covering $\Rcal_\epsilon$ and (possibly) shape-dependent critical values $u_\hb$, the {\em $\eps$-adaptive multiscale scan test} is  
\beq \label{scan_eps}
T_\eps(\yb) = I \Big\{ y[R] > u_\hb, \textrm{ for some } \hb \in [\underline h, \overline h]^d \text{ and } R \in \Rcal(\hb) \cap \Rcal_\eps \Big\}.
\eeq

\subsubsection{Asymptotic theory}

Ideally, we would like to select $\epsilon$ small enough (in fact, decreasing with $n$) that the $\epsilon$-adaptive multiscale scan statistic has asymptotically the same distribution as the (full) adaptive multiscale scan statistic.
As it turns out, it is sufficient to select $\epsilon^{-1}$ on the order of $\sqrt{\log n}$ for this to occur.
We will find that with this choice of $\epsilon$ it is possible to construct an algorithm that can perform an $\epsilon$-covering scan in near-linear time.

Consider critical values of the form \eqref{u_a}-\eqref{tau} and define the following P-value
\beq \label{pval}
\hat \alpha_{n,\hb}(z) = \inf \{ \alpha \in (0,1) :  z \ge u_{n,\hb}(\tau_\alpha)\}.
\eeq
Then the P-value associated with the adaptive multiscale scan test is 
\beq \label{pval_scan}
\hat \alpha_n = \min \Big\{\hat \alpha_{n,\hb}(y[R]) : \hb \in [\underline h, \overline h]^d , R \in \Rcal(\hb) \Big\}.
\eeq
Analogously, the P-value associated with the $\eps$-adaptive multiscale scan test is
\beq \label{pval_eps_scan}
\hat \alpha_{n,\eps} = \min \Big\{\hat \alpha_{n,\hb}(y[R]) : \hb \in [\underline h, \overline h]^d , R \in \Rcal(\hb) \cap \Rcal_\eps\Big\}.
\eeq

\begin{theorem}
\label{thm:epsilon}
Consider the P-value for the multiscale scan or the adaptive multiscale scan, and $\eps$-covering analog, defined in \eqref{pval_scan}-\eqref{pval_eps_scan} respectively. 
Assuming $\epsilon \sqrt{\log n} \rightarrow 0$, we have 
\[
|\hat \alpha_{n,\epsilon} - \hat \alpha_n | = o_\PP(1), \quad n \to \infty.
\]
\end{theorem}

This implies that any such $\epsilon$-scan test enjoys the same asymptotic size and power as the corresponding full scan, established in Theorems~\ref{thm:multiscale_type1} and~\ref{thm:multiscale_type2} for the multiscale scan, and in Theorems~\ref{thm:adaptive_type1} and~\ref{thm:adaptive_type2} for the adaptive multiscale scan. 

\subsubsection{Implementation and computational complexity}

The computational complexity of a scan over an $\eps$-covering depends, of course, on how the $\eps$-covering is designed.
We refer the reader to \citep{MGD,cluster,MR2604703} for some existing examples in the literature.  
We design here another $\eps$-covering which we find easier to scan over in practice. 
Specifically, assuming that $n$ is a power of 2 for convenience, we consider
\beq \label{eq:R_eps}
\Rcal_\epsilon = \bigcup_{\ab \in [ \log_2 n  ]^d }\Big\{ [2^{\ab} \circ \tb, 2^{\ab} \circ (\tb + \fb)] : f_j \in [\lceil 8d / \epsilon^2 \rceil], t_j \in [n / 2^{a_j}], \forall j \in [d]\Big\}. 
\eeq

\begin{proposition} \label{prop:eps-cover}
Suppose that $\epsilon^2 \underline h \rightarrow \infty$ as $n \to \infty$.
When $n$ is large enough, $\Rcal_\eps$ defined in \eqref{eq:R_eps} is indeed an $\eps$-covering of $\Rcal$ for the metric $\delta$ defined in \eqref{delta}.
\end{proposition}

\algref{eps_rect} gives an efficient implementation of a scan over $\Rcal_\epsilon$.
As in \citep{MGD}, we start by summing $y$ over dyadic rectangles, which are defined as rectangles whose side lengths are a power of $2$.
Formally, let ${\rm dyad}_\ab$ denote the result of summing $y$ over all rectangles of shape $2^{\ab}$ with the top-left corner at a multiple of $2^{\ab}$, thought of as a field over the grid $[n 2^{-\ab}]$.
Using dynamic programming, computing $\{\textrm{dyad}_\ab : \ab \in [ \log_2 n  ]^d\}$ can be done in time $O(n^d)$.
This `coarsification' allows us to quickly form spatial approximations to the full spatial scan for a specific shape $\hb$.
Specifically, for a given dyadic scale given by $\ab \in [\log_2 n]^d$ and location and scale given by $t_j \in [n / 2^{a_j}]$, $f_j \in [\lceil 8d / \epsilon^2 \rceil]$ for all $j$, we have 
\[
y\big[[2^{\ab} \circ \tb, 2^{\ab} \circ (\tb + \fb)]\big] = \textrm{dyad}_\ab \big[[\tb, \tb+\fb] \big].
\]

We note that the P-values that appear on Line~\ref{line:pval} of \algref{eps_rect} can be defined in any way, and in particular could be based on other model assumptions.  Put differently, the sole purpose of \algref{eps_rect} is to compute the P-value \eqref{pval_eps_scan} for a given set of critical values in \eqref{pval}, which can be completely arbitrary.

\begin{algorithm}
\caption{Implementation of the $\eps$-adaptive multiscale scan over the $\eps$-covering defined in \eqref{eq:R_eps}.  $n$ is assumed to be a power of 2 for convenience.  The P-values can be as in \eqref{pval} or completely arbitrary, for example based on a different parametric model.}
\label{alg:eps_rect}
\begin{algorithmic}[1] 
\Require Field $\yb$ over $[n]^d$, integers $1 \le \underline h \le \overline h \le n$, $\eps$ such that $\eps^2 \underline h \ge 8d$, P-value functions $\hat \alpha_\hb$
\State Initialize $\textrm{dyad}_\one(\ib) = y(\ib), \forall \ib \in [n]^d$
\For{$\ab \in [\log_2 n]^d \backslash \{1\}^d$} \label{line:start_dyad_loop} 
\State $j' \gets \min \{ j \in [d] : a_j > 1 \}$ 
\For{$\tb \in [n/2^\ab]$}
\State \label{line:dyad} $\textrm{dyad}_\ab(\tb) \gets \textrm{dyad}_{\ab-\eb_{j'}}(\tb \circ (\one + \eb_{j'})) + \textrm{dyad}_{\ab-\eb_{j'}}(\tb \circ (\one + \eb_{j'}) - \eb_{j'})$
\EndFor
\EndFor \label{line:end_dyad_loop}
\State $\underline a \gets \lfloor \log_2 (\epsilon^2 \underline h / (4d)) \rfloor$
\State $\overline a \gets \lceil \log_2 (\epsilon^2 \overline h/(4d)) \rceil$
\State Initialize $\hat \alpha \gets 1$
\For{$\ab \in [\underline a, \overline a]^d$}
\For{$\fb \in [\lceil 8d / \epsilon^2 \rceil]^d$} 
\State \label{line:shat} $\hat s \gets \big( \tprod_j f_j 2^{a_j} \big)^{-\frac 12} \max_{\tb \in [n/2^{\ab}]} (\textrm{dyad}_\ab \ast b_{\fb})(\tb)$
\State \label{line:pval} $\hat \alpha \gets \min\left\{ \hat \alpha, \hat \alpha_{\fb \circ 2^\ab} (\hat s) \right\}$
\EndFor
\EndFor
\Ensure $\hat \alpha$
\end{algorithmic}
\end{algorithm}

\begin{proposition}
\label{prop:alg_eps_net}
Suppose that $\epsilon^2 \underline h \rightarrow \infty$ as $n \to \infty$.
When $n$ is large enough, \algref{eps_rect} performs a scan over $\Rcal_\epsilon$ defined in \eqref{eq:R_eps}.
\end{proposition}

\begin{proposition}
\label{prop:eps_time}
\algref{eps_rect} requires on the order of
$
\max\big\{n^d, \epsilon^{-4d} (n/\underline h)^{d} \log n \big\}
$
basic operations.
\end{proposition}

For example, if $\underline h = n^{a}$ for some fixed $a \in (0,1)$ and $\epsilon = (\log n)^{-1}$ (which is allowed by \thmref{epsilon}), then the computational complexity of $\epsilon$-AdaScan is of order $O(n^d)$, which is precisely linear in the grid size.

\section{Numerical experiments}
\label{sec:experiments}

In this section, we discuss some findings from simulation experiments.
In each of the following experiments, we will generate observations that conform to our assumptions, namely that the random field $y$ is drawn according to \eqref{model} and that there is a rectangular activation under $H_1$ as in \eqref{eq:test}.
We will consider three questions.
\benum
\item For finite $n$, does the adaptive test $T_a(\yb)$ have appreciably superior power compared to the multiscale test $T_m(\yb)$?  
\item For finite $n$, do the theoretically-derived thresholds \eqref{u_m} and \eqref{u_a} control the level of the tests $T_m(\yb)$ and $T_a(\yb)$ as desired?
\item What is the trade-off between computation time and statistical power as we vary $\epsilon$ in the adaptive $\eps$-scan (\algref{eps_rect})? 
\eenum
In all our experiments below, we consider the case of a discrete image ($d=2$) and the signal under the alternative is proportional to the indicator function of a rectangle, i.e., $x(\ib) = \mu / \sqrt{|R^\star|}$ for $\ib \in R^\star$ and $0$ otherwise, for some rectangle $R^\star$.  For the multiscale and adaptive scans, we set $\underline h = 6$.

The first experiment will address the effect that adapting has on the statistical power.
We consider a $256 \times 256$ image ($n = 256$).  We simulate $400$ times from both the null $H_0$ and each instance of the alternative $H_1$.  Under $H_1$, we set $\mu = 6$ and consider  three rectangle sizes ---  $34 \times 81$, $34 \times 38$ and $18 \times 15$ --- with the location of the activation rectangle being chosen uniformly at random.
For each method, we simulate the false discovery rate and the true discovery rate --- the fraction of the $400$ simulations drawn from $H_0$ that were rejected and the fraction from $H_1$ that were rejected, respectively --- and plot them as the parameter $\tau$ varies, producing a receiver operator characteristic (ROC) curve. 
For each rectangle, we compare four methods: the oracle that scans at the scale of $R^\star$, the multiscale scan, the adaptive scan, and a modified adaptive scan based on 
\[
\max_{\hb \in [\underline h, \overline h]^d} \Big(\max_{R \in \Rcal(\hb)} y[R] - v^{\rm mod}_{n, \hb}\Big) v^{\rm mod}_{n, \hb} , \quad {\rm where} \quad v^{\rm mod}_{n, \hb} = \sqrt{2 \tsum_j \log (n / h_j)}.
\]
Notice that $v_{n, \hb}$ is the dominating term in \eqref{v_a}.
Our findings (Figure \ref{fig1}) indicate that the adaptive scan test only marginally outperforms the multiscale scan test, while the modified adaptive test brings a more significantly improvement.
All these tests are closer and closer to the oracle test as size of the activation the rectangle increases.
Because the computational complexity of these tests is $O(n^2)$, evaluating the performance on significantly larger images was not feasible.
The conclusions that we can draw from this are that for images of moderate size, the effects of the lower order terms in the adaptive test inhibits the gains in power that we expect from \thmref{adaptive_type1}.

\begin{figure*}[!htbp]
\centering
\mbox{
\subfigure{\includegraphics[width=2in]{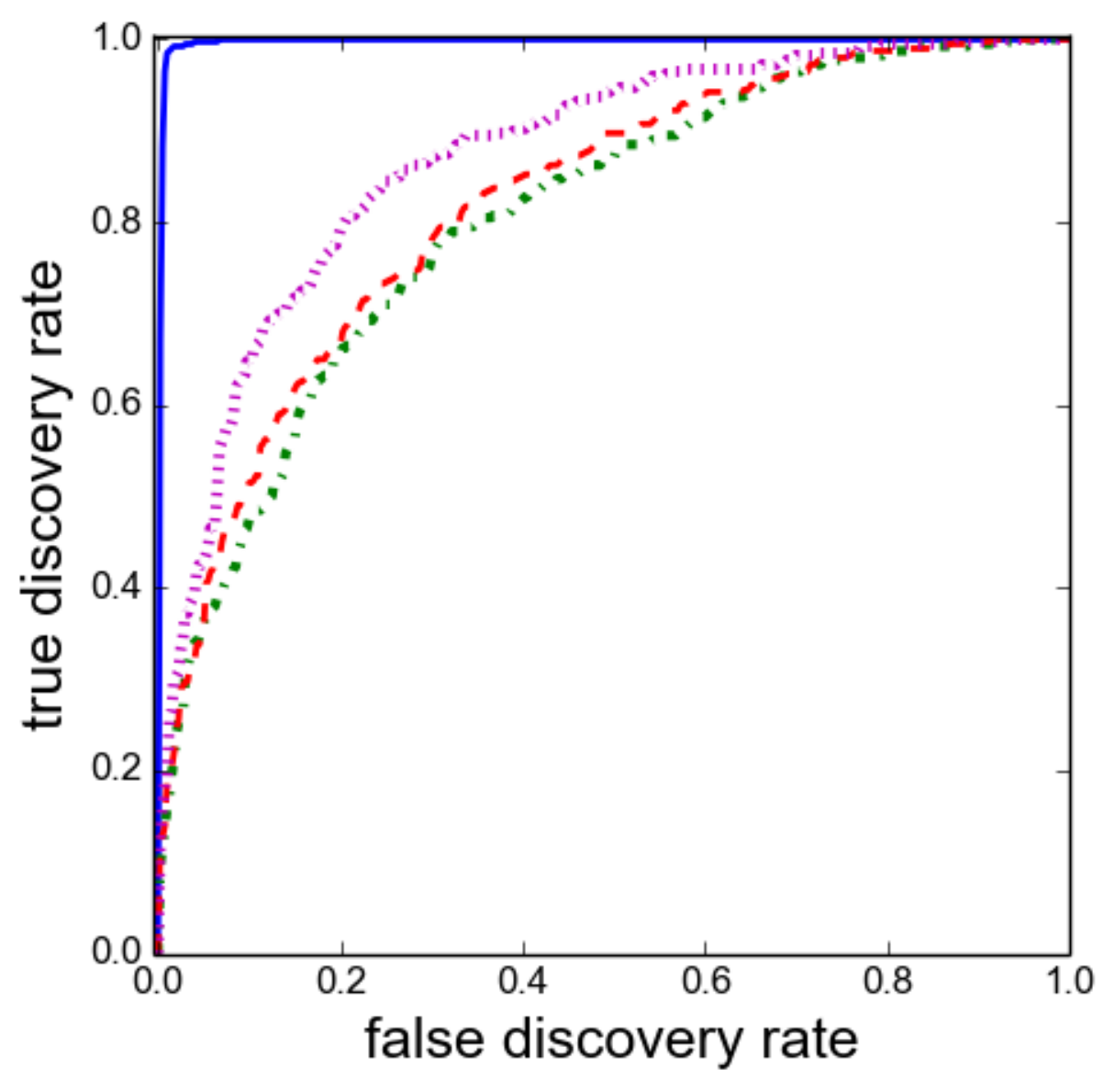}}
\subfigure{\includegraphics[width=2in]{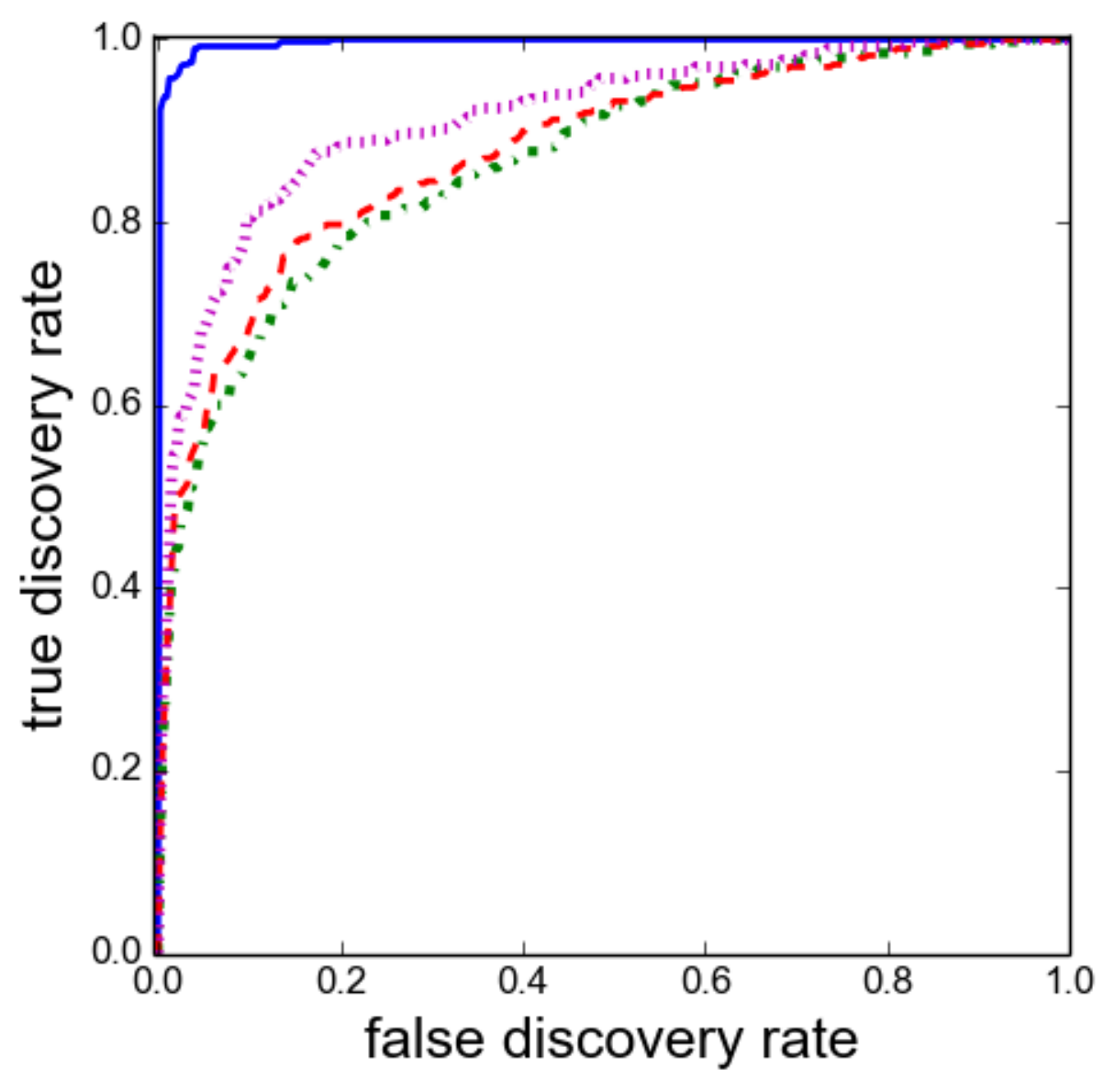}}
\subfigure{\includegraphics[width=2in]{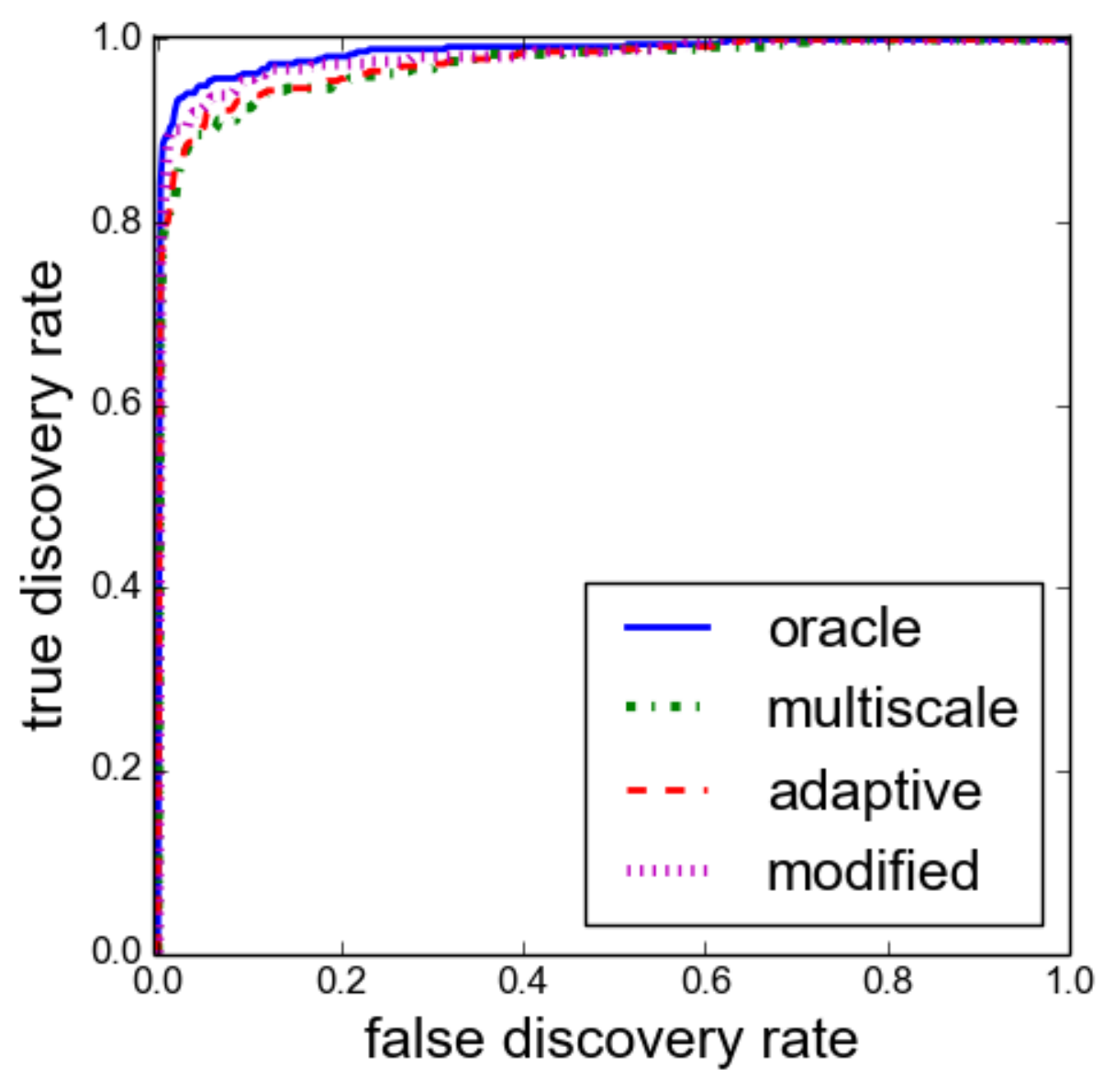}}
}
\caption{{\small (ROC curves for varying rectangle size) The percentage of discoveries that are true versus the percentage that are false, obtained by varying the $\tau$ parameter.  Constructed with $400$ repeats from both $H_0$ and $H_1$, with $n = 256$, $d=2$, $\underline h = 6$, $\mu = 6$, and the rectangle size varying: $34 \times 81$ pixels (left), $34 \times 38$ (middle), $18 \times 15$ (left), with the rectangle location randomized.}
}
\label{fig1}
\end{figure*}

We also provide quantile-quantile plots of the P-value statistics against the uniform distribution on $(0,1)$.
The motivation for this is to assess if the P-values computed based on the thresholds \eqref{u_m} and \eqref{u_a} are accurate.
Our asymptotic theory (Theorems~\ref{thm:oracle_type1}, \ref{thm:multiscale_type1}, and \ref{thm:adaptive_type1}) predicts that this is the case in the large sample limit $n \to \infty$.
We see that the P-values tend to be over-estimated (Figure \ref{fig2}), so that they produce more conservative tests.
In these experiments, we vary the image size to be $128 \times 128$, $256 \times 256$, and $512 \times 512$ --- with $\underline h = 4,6,8$, respectively --- and run $400$ simulations from $H_0$.
In finite samples, we it is clear that our theory provides thresholds that are overly conservative.  
Based on this, In turn we suggest that one uses the adaptive scan P-value as a test statistic and sample from $H_0$ or use a permutation test to from a P-value.

\begin{figure*}[!htbp]
\centering
\mbox{
\subfigure{\includegraphics[width=2in]{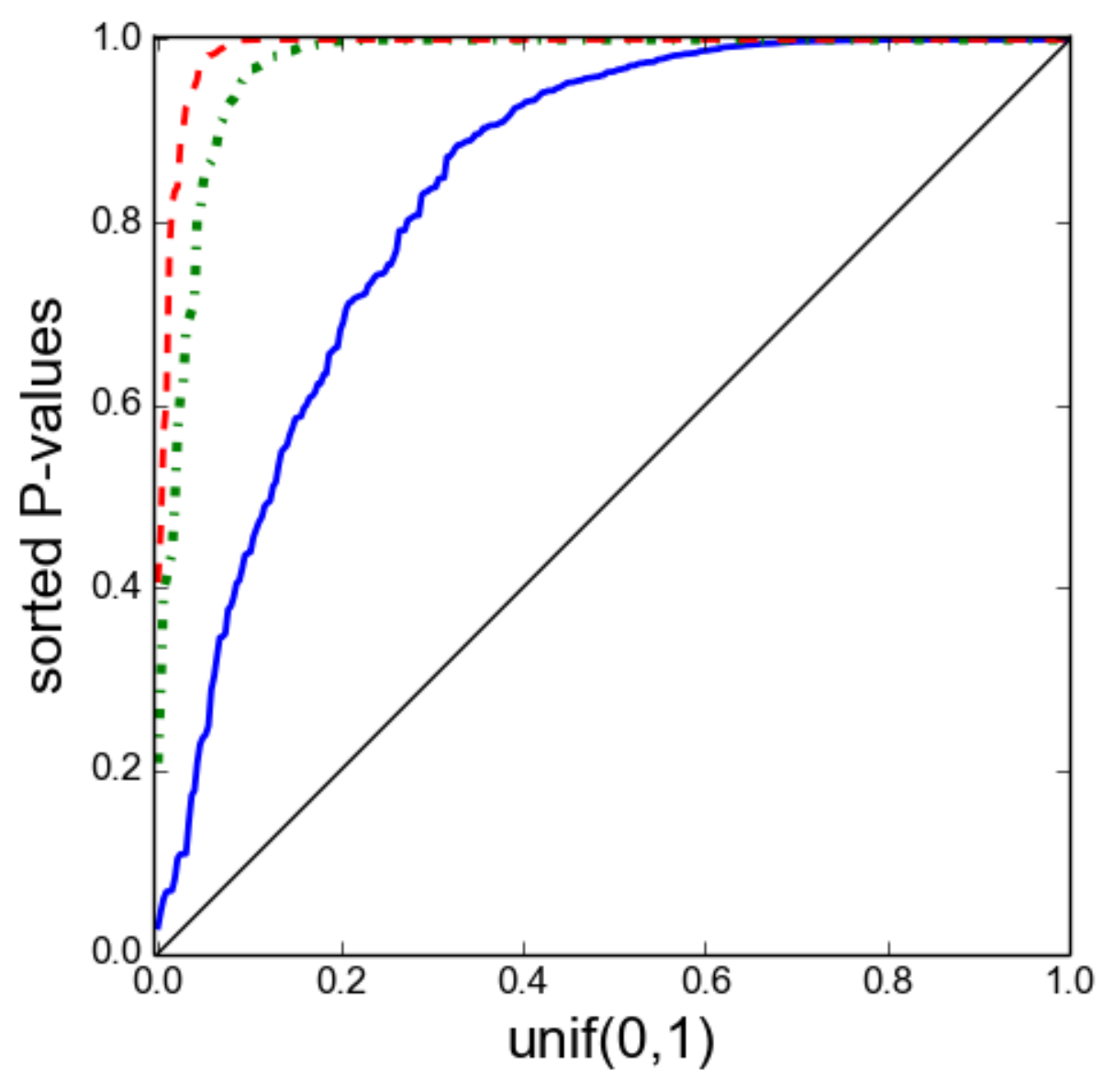}}
\subfigure{\includegraphics[width=2in]{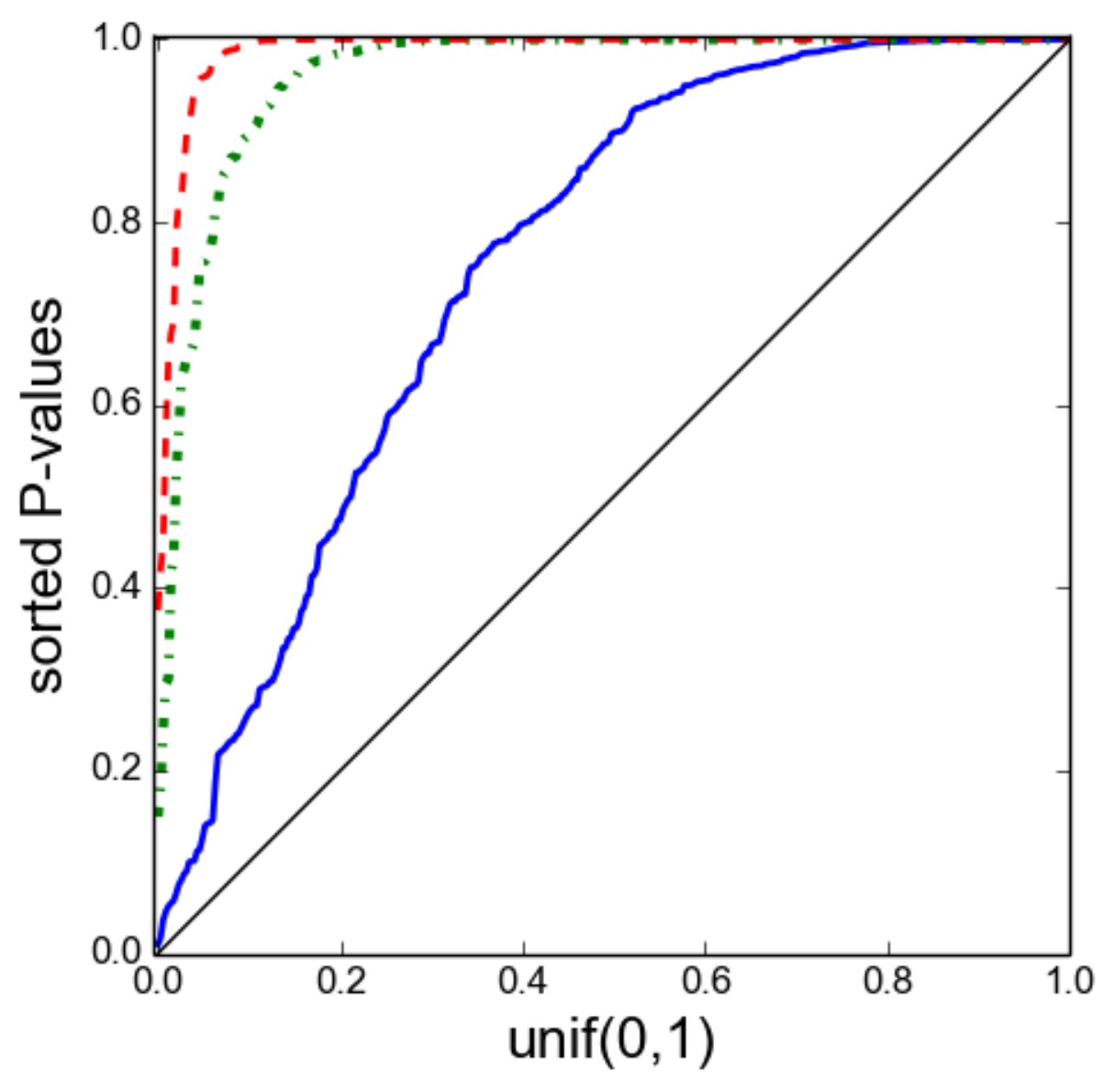}}
\subfigure{\includegraphics[width=2in]{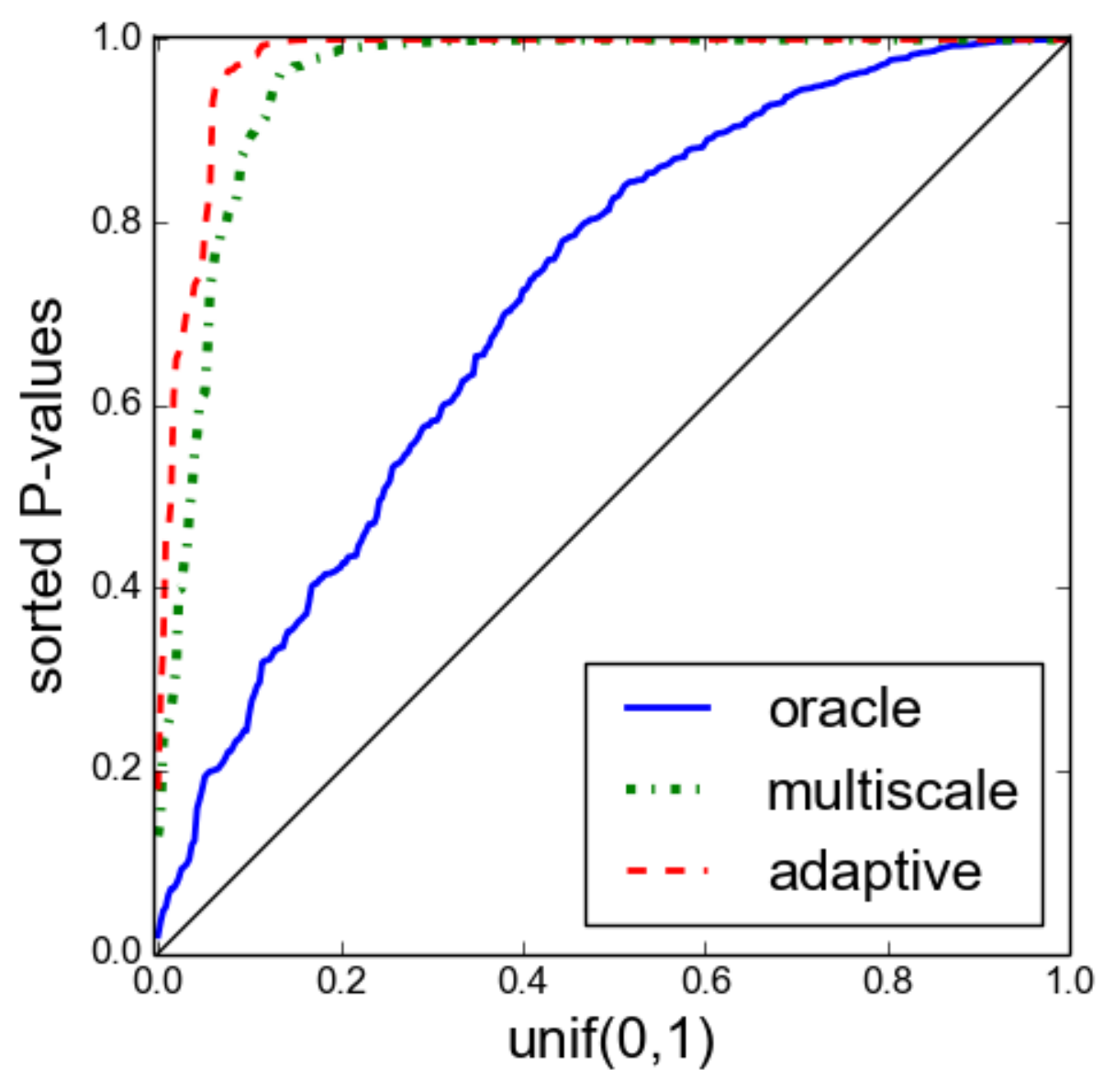}}
}
\caption{{\small (P-value qq-plot) The ordered P-values of $400$ simulations plotted against the quantiles of the uniform$(0,1)$ distribution.  The images size is increasing: $128 \times 128$ (left), $256 \times 256$ (middle), and $512 \times 512$ (right).}
}
\label{fig2}
\end{figure*}

The final set of experiments are intended to demonstrate the performance of \algref{eps_rect}, and highlight the computational and statistical tradeoffs involved.
We derive ROC curves for \algref{eps_rect} by applying it to $480$ simulations with of two different image sizes: $256 \times 256$ and $512 \times 512$ pixels (Figure \ref{fig3}) with parameters $\underline h = 6,12$ and $\mu = 4,5$, and active rectangle of size $61 \times 47$ and $82 \times 35$, respectively.  
We selected the values for $\epsilon$ by making $8d / \epsilon^2$ equal to the integers $1, \ldots, 6$ and selecting from these 4 representative curves.
We interpret the results to mean that as $\epsilon$ decreases, the performance quickly converges to the optimal ROC plot.
We also considered the running time as $\epsilon$ changes (Figure \ref{fig4}).
In this simple experiment we find that, while it is advantageous to have an $\epsilon$ small to increase the power, the improvements in power are generally outweighed by the additional computational burden.

\begin{figure*}[!htbp]
\centering
\mbox{
\subfigure{\includegraphics[width=2in]{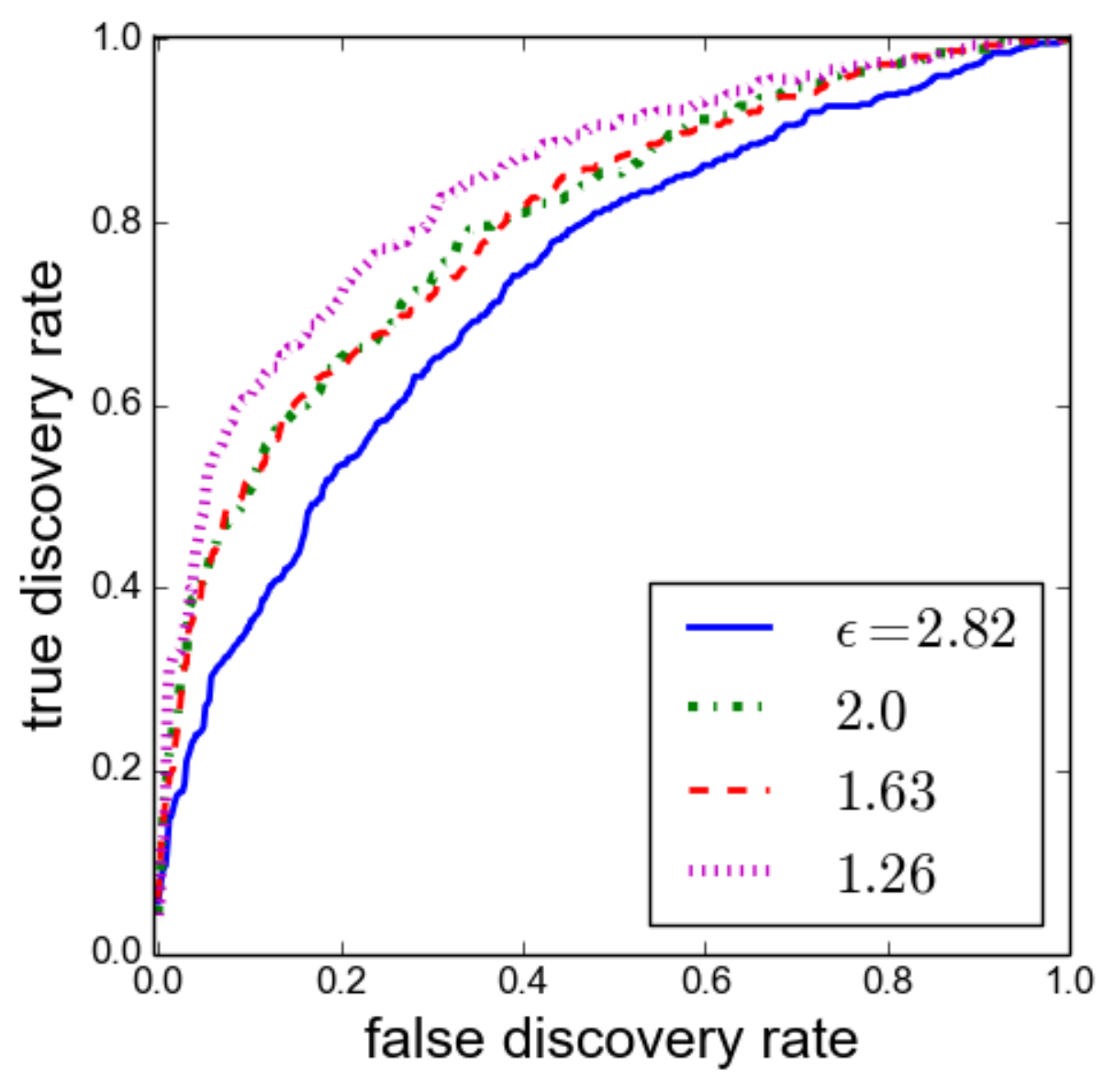}}
\subfigure{\includegraphics[width=2in]{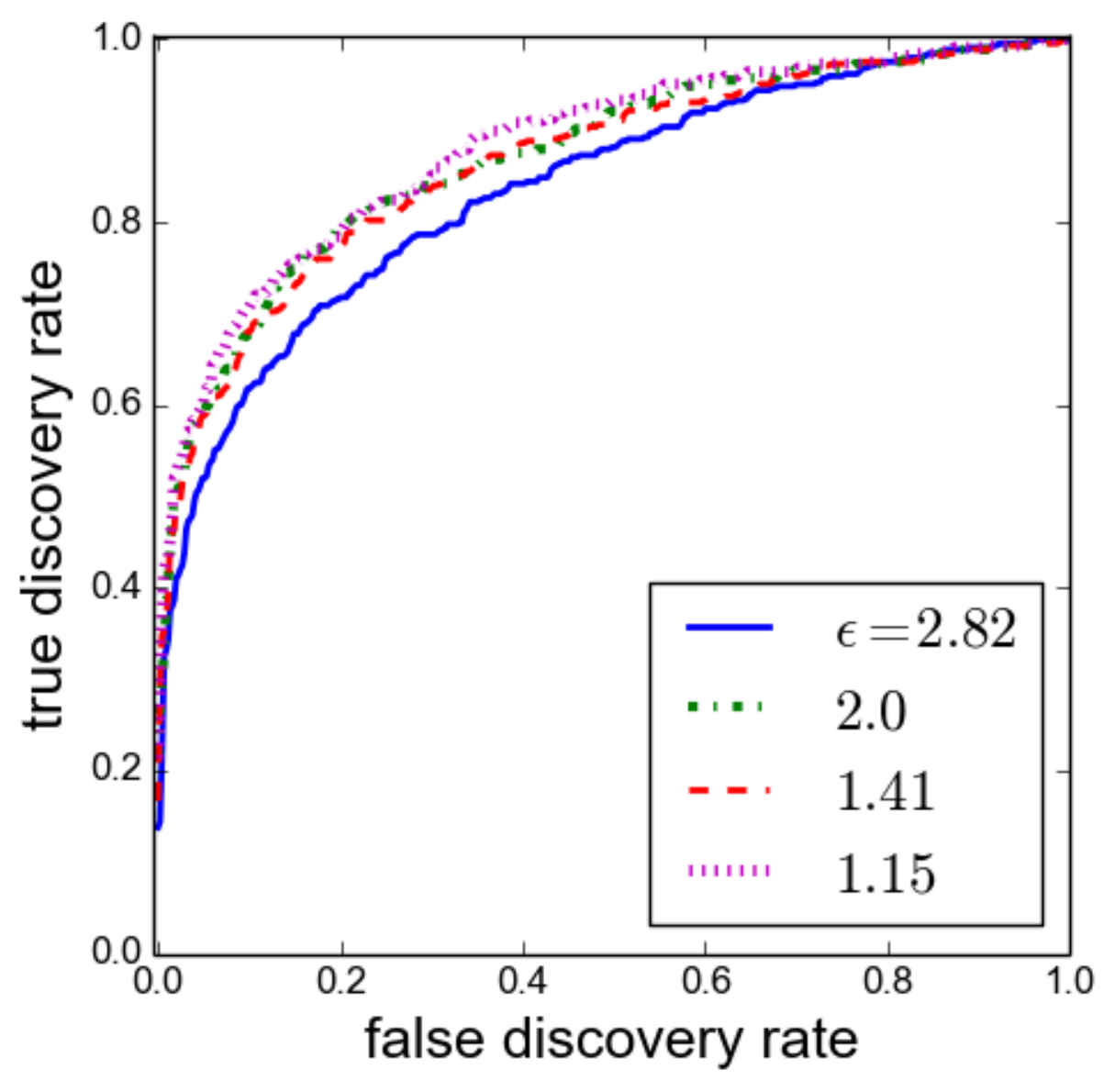}}
}
\caption{\small{(ROC curve for the $\epsilon$-adaptive scan) The ROC curve for \algref{eps_rect} as $\epsilon$ decreases for a $256 \times 256$ image (left) and a $512 \times 512$ image (right) with $\underline h = 6, 12$, $\mu = 4,5$, and a $61 \times 47$ and $82 \times 35$ active rectangle, respectively.  Each setting is repeated $480$ times.}
}
\label{fig3}
\end{figure*}

\begin{figure*}[!htbp]
\centering
\mbox{
\subfigure{\includegraphics[width=2in]{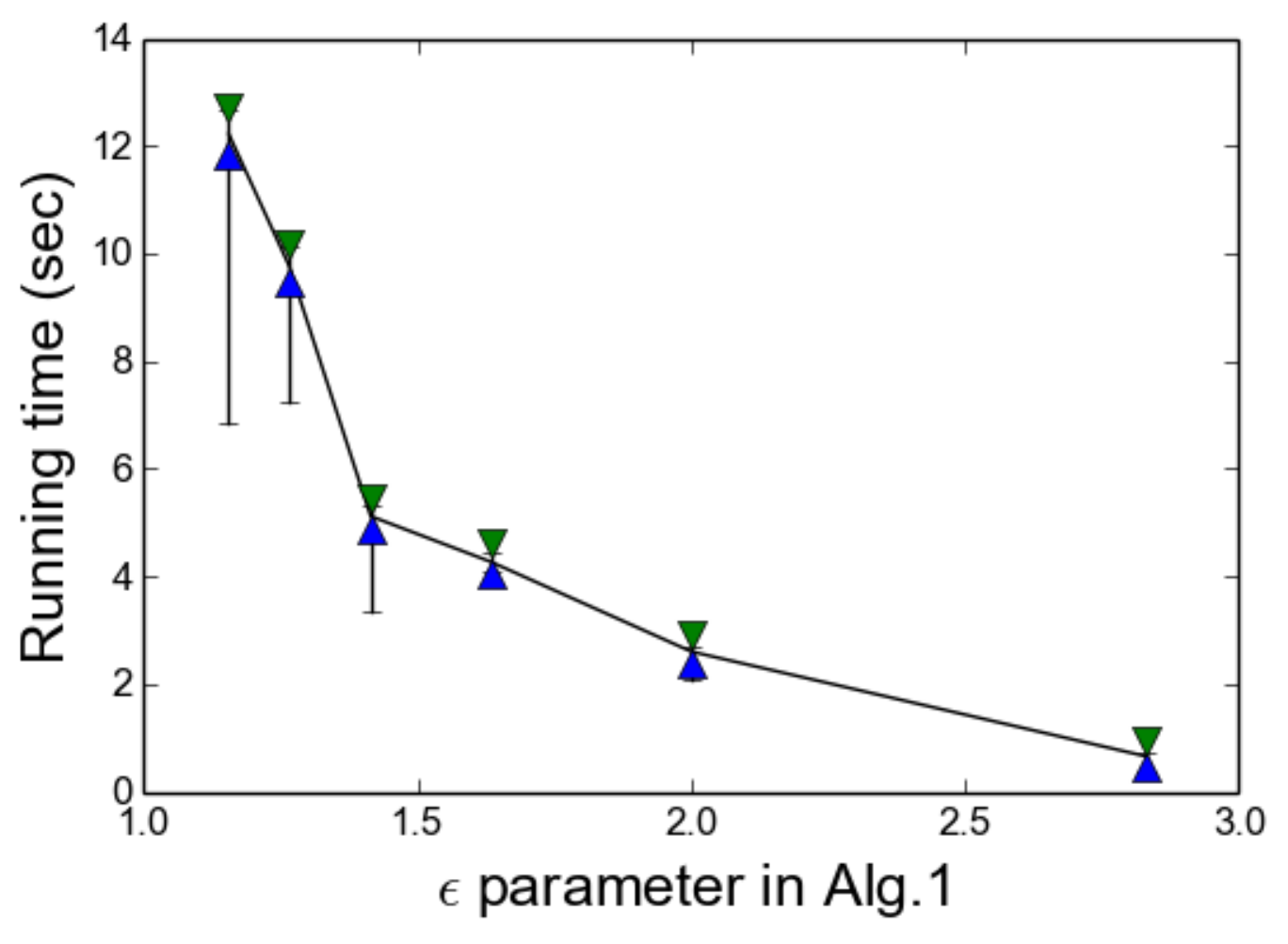}}
\subfigure{\includegraphics[width=2in]{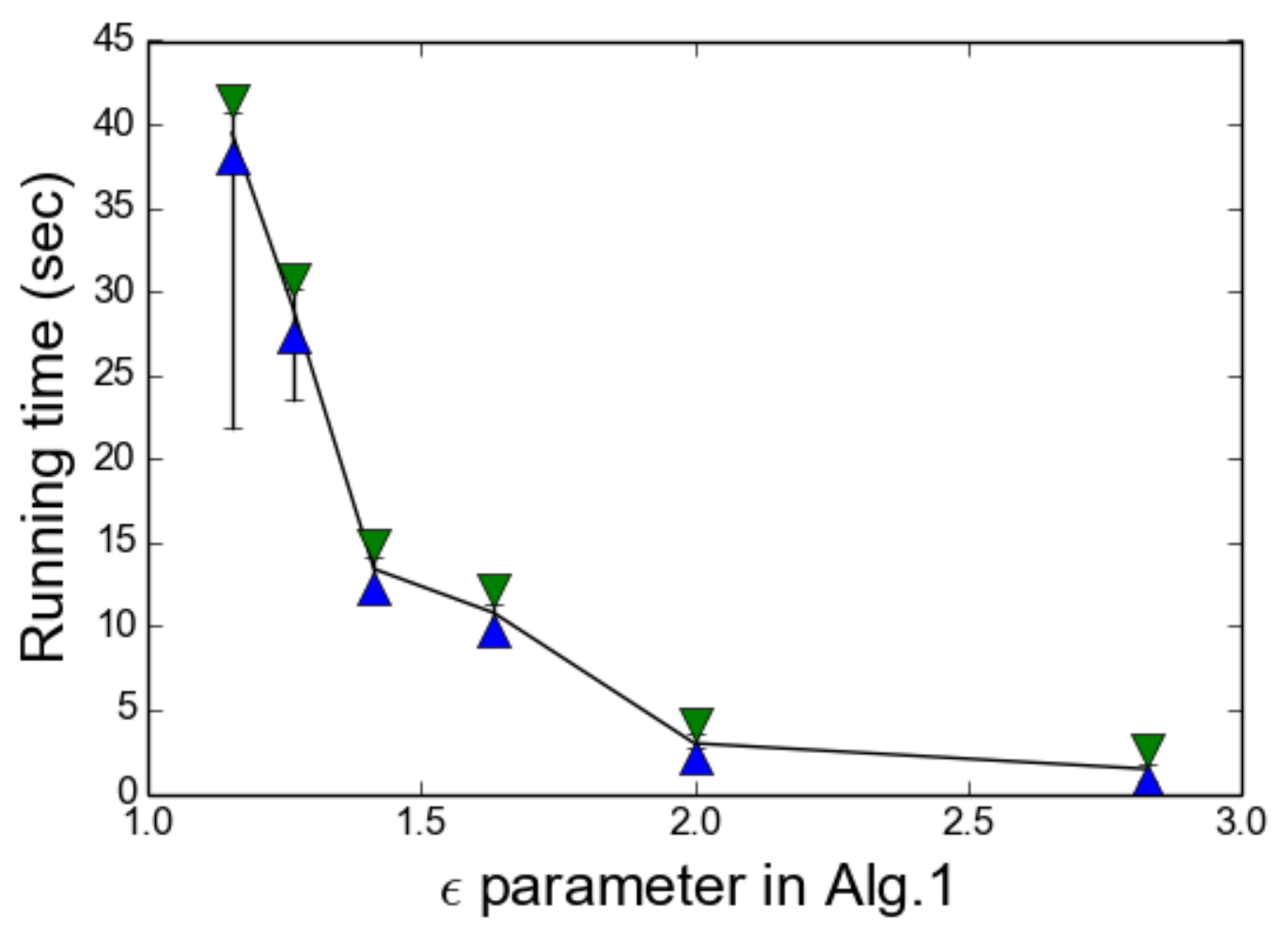}}
}
\caption{\small{(Running time) Same setting as in Figure \ref{fig3}.  Here we plot the running time in seconds on a $2.80$ Ghz virtual CPU as a function of $\epsilon$.  The line indicates the average time, the triangles are the $5$ and $95$ percentiles, and the error bars extend from the minimum to the maximum of the $480$ simulations.}}
\label{fig4}
\end{figure*}

\section{Discussion}
\label{sec:discussion}

We briefly discuss some generalizations and refinements of our work here.

{\em More general signals.}
In this paper we work in a context where the signal has substantial energy over some rectangle of unknown shape and location. 
This motivates the scans over classes of rectangles that we define and study in detail.
Although one could scan over more complicated shapes, to increase power, as done for example in \citep{cluster,MGD,duczmal2006ess,kulldorff2006ess}, the implementation of such scans is generally very complicated and often ad hoc search methods are implemented; also, the asymptotic analyses becomes much more complicated.  
In contrast, scanning over rectangles can be done efficiently and the mathematical analysis can be carried to the exact asymptotics, as we show here --- see also \citep{MGD,MR2604703}.
Moreover, rectangles are building blocks for more complicated shapes and are representative of `thick' or `blob-like' shapes --- see \citep{cluster}.

{\em Signals of arbitrary sign.}
For concreteness and ease of exposition, we consider signals that are implicitly positive over a rectangle.  This can be seen from our definition of Z-score in \eqref{Z}.  This would not be the most appropriate definition when one is expecting signals of arbitrary sign,  for example, when the signal $\xb$ is such that $x(\ib)$ are IID normal with zero mean and variance $\tau$, over some rectangle $R$.  In that case, assuming $R$ is asymptotically large, we have $x[R] \sim \Ncal(0, \tau)$, and is negative with probability a half.
For a sign $\xb = (x(\ib) : \ib \in [n]^d)$ and $R \subset [n]^d$, define $x_2[R] = \sum_{\ib \in R} x(\ib)^2$.
Instead of the class of signals defined in \eqref{alt}, consider
\beq \label{alt-sign}
\Xcal^2_{\tau}(\hb) = \Big\{\xb \in \RR^{n^d} : \min_{R \in \Rcal(\hb)} x_2[R] \ge |R| + \tau \sqrt{2 |R|} \Big\}.
\eeq
In that case, the most natural scans are based on the chi-squared scores $y_2[R], R \in \Rcal$.

{\em Other parametric models.}
Obtaining similar results for other parametric models may be of interest, for example, in epidemiology where the Poisson distribution is used to model counts, and would replace the Gaussian distribution here.  \cite{cluster} extend their first-order analysis to distributions with finite moment generating function, proving that the bounds obtained under normality still apply as long as $\underline h \gg \log n$.  It is possible that a similar phenomenon (essentially due to the Central Limit Theorem) applies at a more refined level, and that our results apply to such distributions, again, as long as $\underline h$ is sufficiently large.

{\em Dependencies.}
A more involved extension of our results would be to allow the observations  $y(\ib), \ib \in [n]^d$ to be dependent.  The results of \cite{chan2006maxima} upon which Kabluchko's arguments (and therefore ours too) are founded apply unchanged to the setting where short-range dependencies are present.  So, in principle, an extension of our work in that direction is possible following similar lines.  But we did not pursue this here for the sake of concreteness and conciseness of presentation.

\section{Proofs}
\label{sec:proofs}

Our method of proof is largely based on \citep{kabluchko2011extremes}, which relies on the work of \cite{chan2006maxima} on the extrema of Gaussian random fields and the Chen-Stein Poisson approximation \citep{arratia1989two}.

{\em Signals that are indicators of rectangles.}
We will focus the remaining of the paper on signals $\xb$ that are proportional to the indicator of a rectangle.  This is asymptotically the most difficult case for all the scans that we consider.  Indeed, we show in this section that the limits in Theorems~\ref{thm:oracle_type2},~\ref{thm:multiscale_type2}, and~\ref{thm:adaptive_type2}, hold when the signal is $\mu |R^\star|^{-1/2} {\bf 1}_{R^\star}$, while for a more signal $\xb$ such that $x[R^\star] \ge \mu$, these are seen to hold as lower bounds when taking the limit inferior.  Together, this shows that the minimax asymptotics stated in Theorems~\ref{thm:oracle_type2},~\ref{thm:multiscale_type2}, and~\ref{thm:adaptive_type2} hold.

We will often leave $n$ implicit, but even then, all the limits are with respect to $n \to\infty$, unless otherwise stated.

\subsection{Locally stationary Gaussian random fields}
\label{sec:local_stationary}
We will begin the proof section with an introduction to some theory for locally-stationary Gaussian random fields (GRFs), particularly their smoothness and extreme value properties.
Throughout this work, we approximate the discrete GRF given by $\{\xi[R], R \in \Rcal\}$ with a continuous version.  
For that, define the continuous analog to $\Rcal$, that is, 
\[
\bar \Rcal = \Big\{ [\tb, \tb+\hb] : \hb \in [\underline h,\overline h]^d, \tb \in [{\bf 0}, n {\bf 1} - \hb] \Big\}.
\]
Let $\Xi$ be a (canonical) Gaussian white noise on $\RR^d$, meaning a random measure on the Borel sets of $\RR^d$ such that, for any integer $k \ge 1$ and any Borel sets $R_1, \dots, R_k$, $\Xi(R_1), \dots, \Xi(R_k)$ are jointly Gaussian, with zero mean and $\Cov(\Xi(R_i), \Xi(R_j)) = \lambda(R_i \cap R_j)$ for all $i,j \in [k]$.
Consider the GRF on $\bar \Rcal$ defined by $\Xi[R] = \Xi(R)/\sqrt{\lambda(R)}$, where $\lambda(R)$ denotes the Lebesgue measure of $R$ when $R$ is a continuous rectangle.
This GRF is denoted $\Xi$ henceforth.  It has zero-mean and covariance structure
\[\Cov (\Xi[R_0], \Xi[R_1]) = \frac{\lambda(R_0 \cap R_1)}{\sqrt{\lambda(R_0)\lambda(R_1)}}, \quad R_0, R_1 \in \bar \Rcal.
\]
Consequently, it is invariant with respect to translations and scalings. 
Following the approach taken by \cite{kabluchko2011extremes}, we approximate the discrete GRF $\xi$ with its continuous counterpart $\Xi$.  
Therefore, we will be interested in the excursion probabilities of $\Xi$, which will require an introduction to locally stationary GRFs.
For convenience, consider the parametrization of the rectangles $\bar \Rcal$ via the one-to-one map $\wb = (\hb,\tb) \mapsto R(\wb) := [\tb, \tb+\hb]$ for $\wb \in (0,\infty)^{2d}$.
We then use the shorthand $\Xi(\wb) = \Xi[R(\wb)]$ for $\wb \in (0,\infty)^{2d}$.
In this way, $\Xi$ can be thought of as a GRF over $(0,\infty)^{2d}$  with the following covariance structure,
\begin{equation}
\label{eq:Xi_covariance}
\Cov(\Xi(\hb,\tb),\Xi(\gb,\sbb)) = \prod_{j=1}^d (h_jg_j)^{-1/2} [(t_j + h_j) \wedge (s_j + g_j) - t_j \vee s_j]_+~,   
\end{equation}
for pairs $(\hb, \tb), (\gb, \sbb) \in (0,\infty)^{2d}$, where $x_+ = x \vee 0$.
Furthermore, define the set of shapes and location that correspond to rectangles in $\bar \Rcal$ as
\[
\Wcal = \Big\{ (\hb, \tb) \in (0,\infty)^{2d} : \hb \in [\underline h,\overline h]^d, \tb \in [{\bf 0}, n {\bf 1} - \hb] \Big\}.
\]
This describes the continuous version of the data under the null distribution $H_0$.
Under the alternative, there is a signal, and the continuous counterpart to the discrete GRF is  
\beq\label{Upsilon}
\Upsilon(\wb) = m(\wb) + \Xi(\wb),
\eeq
where  
\beq \label{m}
m(\wb) = \mu \Cov(\Xi(\wb), \Xi(\wb^\star)),
\eeq
$\wb^\star = (\hb^\star, \tb^\star)$ being the scale and location of the active rectangle.
(Recall that under the alternative we are considering the signal $\mu |R^\star|^{-1/2} {\bf 1}_{R^\star}$.)
We are now prepared to review some relevant results on boundary crossing probabilities of locally-stationary GRFs.

\subsubsection{Boundary crossing probabilities for locally stationary GRFs}

In order to analyze the GRF $\Xi$ we must introduce the notion of local stationarity and the tangent process.
The definitions below are given in \citep{chan2006maxima,qualls1973asymptotic}, and utilized in \citep{kabluchko2011extremes}.

We note that we work in dimension $p = 2d$, except when analyzing the oracle scan, in which case $p=d$, because the shape $\hb^\star$ is given.
Given $K \subset \RR^p$ and $\gamma > 0$, define
\[
[K]_{\gamma} = \{ \wb + \ub : \wb \in K, \| \ub \| \le \gamma \}.
\]
A function $L: \RR_+ \to \RR_+$ is said to be {\em slowly varying} if 
\[
\lim_{x \rightarrow 0} \frac{L(\alpha x)}{L(x)} = 1, \quad \forall \alpha > 0.
\] 
Let $\Scal^{p-1}$ denote the unit sphere in $\RR^p$.
We say that the GRF $\Xi$ is {\em locally stationary over the set $\Wcal$}, if for $\Wcal$ within the domain of $\Xi$, there exists $\alpha \in (0,2]$, $\gamma > 0$, and a slowly varying function $L$, 
such that $[\Wcal]_\gamma \subset (0,\infty)^{2d}$ and for all $\wb \in [\Wcal]_{\gamma}$,
\begin{equation}
  \label{eq:local_stat}
  \EE [\Xi(\wb) \Xi(\wb + \ub)] = 1 - (1 + g_\wb(\ub)) \|\ub\|^\alpha L(\|\ub\|) r_\wb (\ub / \| \ub \|),
\end{equation}
where $r_\wb: \mathbb{S}^{p-1} \to \RR_+$ are continuous functions such that 
\[
\sup_{\vb \in \mathbb{S}^{p-1}} |r_\wb (\vb) - r_{\wb+\ub} (\vb)| \rightarrow 0, \quad \textrm{ as } \ub \rightarrow \zero,
\]
and $g_\wb:\RR^p \to \RR$ are such that 
\[
\sup_{\wb \in [\Wcal]_{\gamma}} |g_\wb(\ub)| \to 0,
\quad \textrm{ as } \ub \rightarrow \zero.
\]
For such processes, the {\em local structure} is defined as
\[
C_\wb (\ub) = \|\ub\|^\alpha L(\|\ub\|) r_\wb (\ub / \| \ub \|),
\]
and we say that the local structure is {\em homogeneous of order} $\alpha$.
Let the {\em tangent process} at $\wb \in \Wcal$ be $\{H_\wb(\ub)\}_{\ub \in \RR^{p}}$, and defined as the GRF satisfying
\[\EE H_\wb(\ub) = -C_\wb(\ub), \quad \ub \in \RR^p,\]
and
\[\Cov (H_\wb(\ub_0), H_\wb(\ub_1)) = C_\wb(\ub_0) + C_\wb(\ub_1) - C_\wb(\ub_0 - \ub_1), \quad \ub_0, \ub_1 \in \RR^p.\] 
The {\em high excursion intensity} is defined as
\[
\Lambda(\wb) = \lim_{m \rightarrow \infty} \frac1{m^p} \EE \exp \Big[ \tsup_{\ub \in [0,m]^p} H_\wb(\ub) \Big]
\]
and has been shown to exist within $(0,\infty)$ in \citep[Lem 5.2]{chan2006maxima}, which in fact states that this convergence is uniform within $\wb \in \Wcal$.
\cite{kabluchko2011extremes} proves the following result by observing that $\Xi$ has the same local structure as a tensor product of normalized differences of Brownian motions.

\begin{lemma}[\cite{kabluchko2011extremes}]
\label{lem:excursion}
The GRF $\Xi$ is locally stationary over $\Wcal$ with $\alpha = 1$ and $L(x) = 1$, with local structure given by
\[
C_{(\hb,\tb)} (\gb,\sbb) = \frac 12 \sum_{j=1}^d \frac{|s_j| + |s_j + g_j|}{h_j},
\]
and high excursion intensity given by 
\[
\Lambda(\hb,\tb) = 4^{-d} \prod_{j = 1}^d h_j^{-2}.
\]
\end{lemma}

Define the function
\[
\psi(x) = \frac{1}{x \sqrt{2\pi}} e^{-\frac 12 x^2}, \quad x \in \RR.
\]

\begin{lemma}[\cite{chan2006maxima} Th 2.1]
\label{lem:chan_constant}
For $K \subset \Wcal \subset \RR^p$ fixed, bounded, and Jordan measurable, and a GRF $\Xi$ that is locally stationary over $\Wcal$ with homogeneity $\alpha$ and high-excursion intensity $\Lambda$,
\[
\PP \big\{\exists \wb \in K : \Xi(\wb) > c \big\} \sim  c^{\frac{2p}{\alpha}}\psi(c) \int_{K} \Lambda(\wb) {\rm d} \wb, \quad \text{as } c \rightarrow \infty.
\]
\end{lemma}

\citep{chan2006maxima} generalized this theorem for calculating the non-constant boundary crossing probability of a locally stationary GRF, which we will use in our analysis of the adaptive scan.
In fact, \citep[Th 2.8]{chan2006maxima} allows a non-constant boundary, a set that is growing with $n$, and holds for non-Gaussian random fields.
We specialize the theorem to our setting.

\begin{lemma}[\cite{chan2006maxima} Th 2.8]
\label{lem:chan_nonconstant}
Let $\Xi$ be a GRF that is locally stationary over $\Wcal$ with homogeneity $\alpha$ and high-excursion intensity $\Lambda$, and take a fixed bounded and Jordan measurable set $K$ such that $[K]_{\gamma} \subset \Wcal$ for some $\gamma > 0$. 
Let $(c_\zeta : \zeta \in (0, 1))$ be a family of real-valued functions defined on $\Wcal$ satisfying
\beq \label{chan_nonconstant1}
\sup_{\wb \in [K]_{\gamma}} c_\zeta(\wb)^{-2p/\alpha} = o(\zeta), \quad \zeta \to 0,
\eeq
for some $\gamma_0 > 0$ fixed and
\beq \label{chan_nonconstant2}
\sup\Big\{c_\zeta(\wb)^2 - c_\zeta(\wb')^2 : \wb, \wb' \in [K]_{2 \zeta}, \|\wb -\wb'\|_\infty \le \zeta\Big\} = o(1), \quad \zeta \to 0.
\eeq
Then 
\[
\PP \big\{\exists \wb \in K : \Xi(\wb) > c_\zeta(\wb) \big\} \sim \int_K c_\zeta(\wb)^{\frac{2p}{\alpha}}\psi(c_\zeta(\wb)) \Lambda(\wb) {\rm d} \wb, \quad \zeta \to 0.
\]
\end{lemma}

\lemref{chan_nonconstant} differs from the statement in \citep[Th 2.8]{chan2006maxima} which includes additional conditions.
This is due to the fact that we assume that $K$ is fixed and $\Xi$ is Gaussian.
Their conditions (2.16) and (2.18) are precisely \eqref{chan_nonconstant1} and \eqref{chan_nonconstant2}, while the condition (2.14) is trivially true for fixed $K$.
In the proof of \citep[Th 2.1]{chan2006maxima} their conditions (A1)-(A5) were shown to hold for locally stationary GRFs, and as a consequence so do (B1)-(B5) since the process is exactly Gaussian and the domain $K$ is fixed.

\subsection{Approximating $\Xi$ with an $\epsilon$-covering}
In this section we state and prove results on the covering properties of $\Wcal$ and the continuity of $\Xi$.  
The metric $\delta$ over $\Rcal$ introduced in \eqref{delta} translates into the following metric on $\Wcal$ (we overload the notation)
\[
\delta(\wb_0,\wb_1) = \delta(R(\wb_0), R(\wb_1)), \quad \forall \wb_0, \wb_1 \in \Wcal.
\] 
An $\epsilon$-covering of $\Wcal$ is defined analogously.  To be sure, it is a subset $\Wcal_\epsilon \subset \Wcal$ such that, for all $\wb \in \Wcal$, there is $\wb' \in \Wcal_\eps$ such that $\delta(\wb, \wb') \le \eps$.
The $\epsilon$-covering number for the metric space $(\Wcal,\delta)$, denoted $N(\Wcal,\delta,\epsilon)$, is the cardinality of a smallest $\epsilon$-covering of $\Wcal$ for $\delta$, and $\log N(\Wcal,\delta,\epsilon)$ is the $\epsilon$-entropy of $\Wcal$.

\begin{lemma} \label{lem:entropy}
For any $0 < \epsilon < \sqrt{4d}$, 
\[
\log N(\Wcal, \delta , \epsilon) \le d \log \left( \frac{32 d^2 n}{\underline h\epsilon^4} \right).
\]
\end{lemma}

\begin{proof}
Let $(\hb,\tb),(\gb,\sbb) \in \Wcal$.
Starting with \eqref{delta} and \eqref{eq:Xi_covariance}, and using the fact that 
\[1 - \prod_{j=1}^d (1 - a_j) \le \sum_{j=1}^d a_j, \quad \text{for any $a_1, \dots, a_d \in [0,1]$},\]
which follows from the union bound or a simple recursion, we have
\begin{eqnarray}
\frac 12 \delta^2 ((\hb,\tb),(\gb,\sbb)) 
&=& 1 - \prod_{j=1}^d (h_jg_j)^{-1/2} [(t_j + h_j) \wedge (s_j + g_j) - t_j \vee s_j]_+ \label{delta-cov} \\ 
&\le& \sum_{j = 1}^d \Big(1 - \frac{1}{\sqrt{h_jg_j}} [(t_j + h_j) \wedge (s_j + g_j) - t_j \vee s_j]_+ \Big)\notag  \\ 
&\le& \sum_{j = 1}^d \Big(1 - \frac{1}{\sqrt{h_jg_j}} [t_j \wedge s_j + h_j \wedge g_j - t_j \vee s_j]_+ \Big)\notag  \\ 
&\le& \sum_{j = 1}^d \frac{1}{\sqrt{h_jg_j}} [\sqrt{h_jg_j} - t_j \wedge s_j - g_j \wedge h_j + t_j \vee s_j]_+ \notag \\ 
&\le& \sum_{j = 1}^d \frac{1}{\sqrt{h_jg_j}} [h_j \vee g_j - t_j \wedge s_j - g_j \wedge h_j + t_j \vee s_j]_+ \notag \\ 
&\le& \sum_{j = 1}^d \theta((h_j, t_j), (g_j, s_j)), 
\quad \text{where } \theta((h, t), (g, s)) = \frac{|h - g| + |t -s|}{\sqrt{hg}}. \label{eq:delta_bound}
\end{eqnarray}
Notice that because
\[
\delta((\hb,\tb),(\gb,\sbb)) \le \sqrt{2 \sum_j \theta((h_j,t_j),(g_j,s_j))},
\]
and $\Wcal \subset [\underline h,n]^d \times [0,n]^d$, it suffices to construct an $(\eps^2/2d)$-covering for each $[\underline h,n] \times [0,n]$ with respect to $\theta$. 
(We define a covering in $\theta$ analogously although it is not necessarily a metric.)
We divide by $\underline h$ everywhere, so that we may focus on $[1,T] \times [0,T]$, where $T = n / \underline h$, by the scale invariance of $\theta$.
Fix $\alpha \in (0,1)$.
We have
\[
[1,T] \times [0,T] \subseteq \bigcup_{k=0}^{[\log T/\log(1/\alpha)]} \bigcup_{\ell=0}^{[1/(\alpha^k (1-\alpha))]} I_k \times I_{k, \ell},
\]
where $I_k = [ \alpha^{k+1} T , \alpha^k T ]$ and $I_{k,\ell} = 
[\ell \alpha^k (1-\alpha) T , (\ell+1) \alpha^k (1-\alpha) T]$.
Take $h,g \in I_k$ and $t,s \in I_{k,\ell}$ for some $k$ and $\ell$ in these ranges.  Then 
\[
\theta((h, t), (g, s)) \le \frac{(1-\alpha) \alpha^k T + (1-\alpha) \alpha^k T}{\alpha^{k+1} T} = \frac{2 (1-\alpha)}{\alpha}.
\]
The cardinality of the resulting covering is equal to 
\[
\sum_{k=0}^{[\log T/\log(1/\alpha)]} ([1/(\alpha^k (1-\alpha))] + 1)
\le \frac{T}{(1-\alpha)^2} + \frac{\log T}{\log (1/\alpha)} \le \frac{2T}{(1-\alpha)^2},
\]
using the fact that $\log(1+x) \ge x^2 \log(2)$ for all $x \in [0,1]$ and $t \ge \log(t)/\log(2)$ for all $t \ge 1$.  
When we choose $\alpha = 4d/(4d + \epsilon^2)$, the tensor product of these coverings, repeated over $j=1,\dots,d$, is an $\epsilon$-covering of $\Wcal$, of cardinality
\[
\Big(\frac{2T}{(1-\alpha)^2}\Big)^d = (2 T (4d/\epsilon^2)^2)^d,
\]
since $1-\alpha \le \epsilon^2/(4d)$.
\end{proof}

We use this bound on the entropy and a continuity property of $\Xi$ to bound $\Xi$.
This bound will be crude relative to the asymptotic guarantees which are the focus of this work, but is necessary as a lemma.
For an $\epsilon$-covering, $\Wcal_\epsilon \subset \Wcal$, define the {\em interpolated GRF} $\Xi_\epsilon$ over $\Wcal$, with value at $\wb \in \Wcal$ given by $\Xi_\epsilon(\wb) = \Xi(\wb')$, where $\wb' = \argmin\{ \delta(\wb_0,\wb) : \wb_0 \in \Wcal_\epsilon \}$; if the minimizer is not unique, then choose a minimizer arbitrarily.
For a real-valued function $f$ over $\Wcal$, let $\|f\|_\infty = \sup_{\wb \in \Wcal} |f(\wb)|$.

\begin{lemma}
\label{lem:xi_bound}
Consider the GRF $\Xi$ introduced in \secref{local_stationary}.  
In our context, it has the following properties.
\begin{enumerate}
\item The supremum of $\Xi$ has the following behavior
\[
\|\Xi\|_\infty = O_\PP\left( \sqrt{ \log(n/\underline h)} \right).
\]
\item Let $\Ucal \subset \Wcal$ be such that there exists a constant $C > 0$ with the property that $\max_j |t_j - s_j| \le C$ and $\max_j |\log h_j - \log g_j| \le C$ for all $(\hb,\tb), (\gb,\sbb) \in \Ucal$.
Then 
\[
\sup_{\wb \in \Ucal} \Xi(\wb) = O_\PP ( 1 ).
\]
\item Let $\Xi_\epsilon$ be an interpolated GRF built on an $\epsilon$-covering of $\Wcal$ where $\epsilon < 1$.  Then
\[
\left\|\Xi - \Xi_\epsilon\right\|_\infty = O_\PP \left(\epsilon \sqrt{ \log (n/(\underline h \epsilon^4))}\right).
\]
\end{enumerate}
\end{lemma}

\begin{proof}
We prove each part in turn.

{\em Part 1.} 
Let $T = n / \underline h$.
Note that $\delta (\wb,\wb') \in [0,\sqrt 2]$.  Dudley's metric entropy theorem \citep[Th 6.1.2]{marcus2006markov}, along with \lemref{entropy}, can be applied to show that

\[
\EE (\|\Xi\|_\infty) 
\le 16 \sqrt 2  \int_0^{\sqrt{2}} \sqrt{ \log N(\Wcal, \delta , \epsilon) } {\rm d}\epsilon 
\le 16 \sqrt 2 \int_0^{\sqrt 2} \sqrt{ d \log \left(32 d^2 T/\epsilon^4\right) } {\rm d} \epsilon 
= O \left(\sqrt{ \log T } \right).
\]
The result now follows by Markov's inequality.

{\em Part 2.} 
This can be proven by noticing that for any $\epsilon$, the entropy of $\Ucal$ satisfies
\[
\log N(\Ucal,\delta ,\epsilon) \le C_0 \log (1/\epsilon)
\]
for some constant $C_0$, using a construction analogous to that used in the proof of \lemref{entropy}.
The rest follows as in the proof of Part 1.

{\em Part 3.} 
As before let $T = n / \underline h$.
By the definition of $\Xi_\epsilon$,
\[
\EE \left( \left\|\Xi_\epsilon - \Xi \right\|_\infty \right) \le \EE \Big( \sup \{ \Xi(\wb_1) - \Xi(\wb_2) : \wb_1, \wb_2 \in \Wcal : \delta(\wb_1,\wb_2) \le \epsilon\} \Big).
\]
Applying Dudley's theorem and \lemref{entropy}, we bound the RHS by
\[
99 \int_0^\epsilon \sqrt{ \log N(\Wcal, \delta , \eta) } {\rm d}\eta 
\le 99 \epsilon \int_0^1 \sqrt{ d \log \left(32 d^2 T/\epsilon^4\right) + 4d \log \left(1/\eta\right)} {\rm d} \eta 
= O \left( \epsilon \sqrt{ \log \left(T/\epsilon^4\right) } \right).
\]
The result follows by Markov's inequality.
\end{proof}

We will now analyze the P-values resulting from our various scan statistics by their Lipschitz property.
This will allow us to demonstrate that if $\epsilon$ is decreasing quickly enough, the P-value of each test when evaluated over an $\epsilon$-covering is asymptotically indistinguishable from the P-value when evaluated over the entire set $\Wcal$.
In the end, we will have proven \thmref{epsilon}, but these results will also be useful to prove other results.  
For convenience, we work with $\tau$ instead of $\alpha$, related by \eqref{tau}.   For each scan statistic, let $\hat \tau$ be the value of $\tau$ such that the scan statistic equals its threshold (\eqref{u_o}, \eqref{u_m}, or \eqref{u_a}).  It takes the form
\begin{eqnarray}
\hat \tau 
&=& \max_{\wb \in \Wcal'} a(\wb) \left( y[R(\wb)] - a(\wb) \right) + b(\wb) \notag \\
&=& \max_{\wb \in \Wcal'} a(\wb) \left( \xi[R(\wb)] + m(\wb) - a(\wb) \right) + b(\wb), \label{eq:hat_tau}
\end{eqnarray}
where $m$ is defined in \eqref{m} (with $m \equiv 0$ under $H_0$), while $a$, $b$ and $\Wcal' \subset \ZZ^{2d} \cap \Wcal$ will depend on which scan statistic we are considering.
In all cases, 
\[\sqrt{2d} \le \sqrt{2d \log(n/\overline h)} \le a(\wb) \le \sqrt{2d \log(n/\underline h)}, \quad \forall \wb \in \Wcal'.\]
We will relate the statistic $\hat \tau$ with the random variable
\begin{equation}
\label{eq:tilde_tau}
\tilde \tau = \max_{\wb \in \Wcal} a(\wb) \left( \Xi(\wb) + m(\wb) - a(\wb) \right) + b(\wb).
\end{equation}

\begin{lemma}
\label{lem:approx_tau}
Suppose there are constants $L > 0$ and $\epsilon_0 > 0$ such that
\beq \label{ab}
\big|a(\wb_1) -a(\wb_2)\big| \vee \big|b(\wb_1) -b(\wb_2)\big| \le L \delta (\wb_1,\wb_2), \quad \forall \wb_1, \wb_2 \in \Wcal : \delta (\wb_1, \wb_2) \le \eps_0.
\eeq
Then $|\hat \tau - \tilde \tau| = o_\PP(1)$ if $\Wcal'$ is an $\epsilon$-covering of $\Wcal$ with 
\beq \label{ab_eps}
\epsilon \big( \mu + \sqrt{\log \left(n/(\underline h \epsilon^4)\right)}\big) \sqrt{\log(n/\underline h)} = o(1).
\eeq
\end{lemma}

\begin{proof}
Applying the triangle inequality,
\beqn
|\hat \tau - \tilde \tau| 
&\le& \max_{\wb \in \Wcal} \min_{\wb' \in \Wcal'} | a(\wb) \left( \Xi(\wb) + m(\wb) - a(\wb) \right) - a(\wb') \left( \Xi(\wb') + m(\wb') - a(\wb') \right) | \\
&& + \max_{\wb \in \Wcal} \min_{\wb' \in \Wcal'} |b(\wb) - b(\wb')|.
\eeqn
For the second term, we use the fact that $b$ is Lipschitz and that $\Wcal'$ is an $\epsilon$-covering, to get
\[
\max_{\wb \in \Wcal} \min_{\wb' \in \Wcal'} |b(\wb) - b(\wb')| \le \max_{\wb, \wb' \in \Wcal : \delta (\wb, \wb') \le \epsilon} |b(\wb) - b(\wb')| \le L \epsilon.
\]
For the first term, it is bounded by 
\beqn
&& \max_{\wb \in \Wcal}  \left( \Xi(\wb) + m(\wb) - a(\wb) \right) \min_{\wb' \in \Wcal'} | a(\wb) - a(\wb') | \\
&& + \max_{\wb \in \Wcal} a(\wb) \min_{\wb' \in \Wcal'} \left[| \Xi(\wb) - \Xi(\wb') | + | m(\wb) - m(\wb') | + | a(\wb) - a(\wb') | \right].
\eeqn
We have
\[
\min_{\wb' \in \Wcal'} | a(\wb) - a(\wb') | \le L \epsilon, \quad \forall \wb \in \Wcal,
\]
by the fact that $a$ is Lipschitz and $\Wcal'$ is an $\epsilon$-covering for $\Wcal$; we have
\[
\max_{\wb \in \Wcal}  \left( \Xi(\wb) + m(\wb) - a(\wb) \right) \le \max_{\wb \in \Wcal} \Xi(\wb) + \mu = O_\PP\left( \sqrt{ \log(n/\underline h)} \right) + \mu,
\]
by \lemref{xi_bound}, the fact that $m(\wb) \le \mu$ and $a(\wb) \ge 0$ for any $\wb \in \Wcal$; we have
\[\max_{\wb \in \Wcal} a(\wb) \le \sqrt{2d \log(n/\underline h)},\]
as well as 
\[
\min_{\wb' \in \Wcal'} \left[| \Xi(\wb) - \Xi(\wb') | + | m(\wb) - m(\wb') | + | a(\wb) - a(\wb') | \right] \le O_\PP \left(\epsilon \sqrt{ \log (n/(\underline h \epsilon^4)}\right) + \mu \epsilon / 2 + L \epsilon,
\]
for all $\wb \in \Wcal$, by \lemref{xi_bound}.
From this, we conclude.
\end{proof}

For the oracle and multiscale scan statistics, $a$ and $b$ are constant in $\wb$ and so they are trivially Lipschitz.
For the adaptive multiscale scan, we verify below that they indeed satisfy~\eqref{ab}.

\begin{lemma}
\label{lem:ada_lipschitz}
For the adaptive multiscale scan, based on \eqref{v_a}, we have $a((\hb,\tb)) = v_{n,\hb}$ and $b((\hb,\tb)) = -\kappa - (4d - 1) \log \left( v_{n,\hb} \right)$, and they satisfy \eqref{ab} for some $L > 0$ and $\eps_0 > 0$ depending only on $d$.
\end{lemma}

\begin{proof}
Let $f(\hb) = v_{n,\hb}^2/2$.  Since $2 f(\hb) \ge 1$ and $\log x$ has derivative bounded by $1$ over $[1,\infty)$, it is sufficient to show that $(\hb,\tb) \to f(\hb)$ is Lipschitz with respect to $\delta$.
From \eqref{delta-cov}, we see that 
$\delta^2 ((\hb, \tb), (\gb, \sbb)) \ge \delta_\ddag(\hb, \gb)$, where 
$\frac12 \delta_\ddag(\hb, \gb) = 1 - \prod_j (h_j \wedge g_j)/\sqrt{h_j g_j}$, so that we may work with $\delta_\ddag$ instead of $\delta^2$.
By the fact that $\log$ has derivative bounded by 1 on $[1,\infty)$, 
\begin{eqnarray*}
|f(\hb) - f(\gb)| 
&\le& \sum_{j = 1}^d |\log h_j - \log g_j | + 2 \sum_{j = 1}^d |\log (1 + \log (h_j / \underline h)) - \log (1 + \log (g_j/\underline h)) | \\
&\le& 3 \sum_{j = 1}^d |\log h_j - \log g_j |,
\end{eqnarray*}
with
\[ \sum_{j = 1}^d |\log h_j - \log g_j | = 2 \sum_{j=1}^d \log \frac{\sqrt{h_j g_j}}{h_j \wedge g_j} = -2 \log \big(1-\tfrac12 \delta_\ddag(\hb, \gb)\big) \le 2 \delta_\ddag(\hb, \gb), \]
when $\delta_\ddag(\hb, \gb) \le 1$.  
Finally, if $\delta((\hb,\tb),(\gb,\sbb)) \le \epsilon_0$ as in \eqref{ab} then $\delta^2((\hb,\tb),(\gb,\sbb)) \le \epsilon_0 \delta((\hb,\tb),(\gb,\sbb))$.

\end{proof}

The following lemma allows us to approximate a discrete scan with its continuous counterpart.

\begin{lemma}
\label{lem:Z_epsnet}
$\Wcal' = \Wcal \cap \ZZ^{2d}$ is a $(\sqrt{4d/\underline h})$-covering for $\Wcal$ with respect to $\delta$.
\end{lemma}

\begin{proof}
Let $\wb = (\hb,\tb) \in \Wcal$ and define $\wb' = (\lfloor h_1 \rfloor, \ldots, \lfloor h_d \rfloor, \lfloor t_1 \rfloor, \ldots, \lfloor t_d \rfloor)$, which is in $\Wcal'$ by construction.  By \eqref{eq:delta_bound}, and recalling that $\underline h$ is an integer,
\[
\delta^2 (\wb,\wb') \le 2 \sum_{j = 1}^d \frac{|h_j - \lfloor h_j \rfloor| + |t_j - \lfloor t_j \rfloor|}{\sqrt{h_j \lfloor h_j \rfloor}} \le \frac{4 d}{\underline h}. 
\]
\end{proof}

\subsection{Proofs: main results}

The following lemma will allow us to derive the asymptotic threshold from the excursion probabilities that we will derive in the following proofs.

\begin{lemma}
\label{lem:u_rate}
Let $s$ and $t$ be constants, let $(\eta_m)$ be a sequence tending to 0, and define
\beq\label{u_rate}
u_m = \sqrt{2 \log m} + \frac{s \log (\sqrt{2 \log m}) + t +\eta_m}{\sqrt{2 \log m}}.
\eeq
Then
\[
e^t m \, u_m^s e^{- \frac 12 u_m^2}  = 1 + O\Big(\eta_m + \frac{(\log\log m)^2}{\log m}\Big), \quad m \to \infty.
\]
\end{lemma}

\begin{proof}
We have
\[\begin{split}
u_m^2 &= 2 \log m + 2 t + s \log (2 \log m) + O\Big(\eta_m + \frac{(\log\log m)^2}{\log m}\Big), \\
\log u_m &= \tfrac12 \log(2 \log m) + O\Big(\frac{\log\log m}{\log m}\Big).
\end{split}\]
From this, we get that
\[
\log\Big(e^t m\, u_m^s e^{- \frac 12 u_m^2}\Big) =  O\Big(\frac{\log\log m}{\log m}\Big) + O\Big(\eta_m + \frac{(\log\log m)^2}{\log m}\Big) = O\Big(\eta_m + \frac{(\log\log m)^2}{\log m}\Big),
\]
and the result follows by applying the exponential.
\end{proof}


\subsubsection{Proof of Theorem \ref{thm:oracle_type1}}
\label{sec:proof_oracle_type1}
Since the shape $\hb^\star$ is given, in this section we let $\Xi(\tb) = \Xi(\hb^\star, \hb^\star \circ \tb)$, indexed only by the spatial parameter $\tb \in \Tcal = \times_{j = 1}^d [0,T_j]$ where $T_j = n / h^\star_j$.
This is after rescaling, where we divided the $j$th coordinate by $h^\star_j$.
Specifically, the reparametrized GRF has zero mean and covariance structure,
\[
\Cov ( \Xi(\tb), \Xi(\tb')) = \frac{\lambda(R(\hb^\star,\tb \circ \hb^\star) \cap R(\hb^\star, \tb' \circ \hb^\star))}{\sqrt{\lambda(R(\hb^\star,\tb \circ \hb^\star))\lambda(R(\hb^\star, \tb' \circ \hb^\star))}} = \lambda(R(\tb) \cap R(\tb')),
\]
where $R(\tb) := [\tb, \tb + {\bf 1}]$. 
The GRF $\Xi$ restricted to $\Tcal$ is stationary, thus it is locally stationary over $\Tcal$, but in $p = d$ dimensions.
Moreover, it has the local structure $C_\tb(\sbb) = \| \sbb \|_1$, by evaluating the local structure in \lemref{excursion} to the case in which $\hb = \one$ and $\gb = \zero$.
Hence, we know that it is homogeneous of order $\alpha = 1$ with $L = 1$ and $r_\tb(\ub/\|\ub\|) = \| \ub \|_1 / \|\ub\|$.
Due to the restriction to $\Tcal$, the tangent process of $\{\Xi(\tb)\}_{\tb \in \Tcal}$ must also be restricted $\Tcal$.
This will alter the high-excursion intensity from that given in \lemref{excursion}, which we derive next.

In order to prove \lemref{excursion}, \cite{kabluchko2011extremes} developed a technique for analyzing the tangent process using sums of independent Brownian motions.
We use the same approach.  First, note that a version of the tangent process is given by\[
U(\sbb) = \sum_{j=1}^d \sqrt 2 V_j(s_j), \quad \sbb = (s_1, \dots, s_d) \in \RR_+^d,
\] 
where $V_j$ are independent versions of the standard Brownian motion with drift $-|s_j|/\sqrt 2$.
(Notice that, when calculating the high excursion intensity $\Lambda$, the tangent process is restricted to the positive orthant.)
To see that $U$ is indeed a version of the tangent process, 
notice that, for all $\sbb, \sbb' \in \RR^d_+$, $\EE [U(\sbb)] = -\|\sbb\|_1$ and 
\[
\Cov(U(\sbb),U(\sbb')) = 2 \sum_{j=1}^d \Cov(V_j(s_j), V_j(s_j')) = 2\sum_{j=1}^d s_j \wedge s_j' = \|\sbb\|_1 + \|\sbb'\|_1 - \|\sbb - \sbb'\|_1~.
\]
Evaluating $\Lambda$,
\[
\Lambda = \lim_{T \rightarrow \infty} \frac{1}{T^d} \EE \exp \left( \tsup_{\sbb \in [0,T]^d} U(\sbb) \right)\\
= \left[ \lim_{T \rightarrow \infty}  \frac{1}{T} \EE \exp \left( \tsup_{s \in [0,T]} \sqrt 2 V_1(s) \right) \right]^d = H_1^d,\\
\]
where $H_1 = 1$ is Pickands constant for $\alpha = 1$ \citep{pickands1969upcrossing}.
We may now apply \lemref{chan_constant}, and the high excursion probability becomes
\beq \label{high_ex}
\PP \Big\{ \tsup_{\tb \in K} \Xi(\tb) > u \Big\} \sim \frac{\lambda(K)}{\sqrt{2\pi}} u^{2d - 1} e^{-u^2 / 2}.
\eeq
Until further notice, we take $u$ to be the critical value \eqref{u_o}.
Recall that $\lambda$ is the Lebesgue measure, here in $\RR^d$.
Define
\[
\overline \Ical = \times_{j=1}^d [\lfloor T_j \rfloor], \quad \underline \Ical = \times_{j=1}^d [\lfloor T_j \rfloor - 1].
\]
Consider the events $E_\ib = \big\{ \sup_{\tb \in R(\ib)} \Xi(\tb) > u \big\}$ for $\ib \in \overline \Ical$.
Notice that by translational invariance,
\beq \label{translation}
\forall \ib \in \overline \Ical, \quad \PP ( E_{\ib}) = \PP ( E_{\zero} ),
\eeq
where, applying \eqref{high_ex}, 
\beq \label{high_ex0}
\PP ( E_{\zero} ) = \PP \Big\{ \tsup_{\tb \in R(\zero)} \Xi(\tb) > u \Big\} \sim \frac{\lambda(R(\zero))}{\sqrt{2\pi}} u^{2d - 1} e^{-u^2 / 2} \sim e^{-\tau} \prod_{j=1}^d T_j^{-1} ,
\eeq
where the second equivalence comes from by applying \lemref{u_rate}.

We will now establish a Poisson limit for the above process over the entire set $\Tcal$ based on finite range dependence.
Two events, $E_{\ib}, E_{\ib'}$, are independent if $|i_j - i_j'| > 1$, for some $j \in [d]$.
Consider thus the `blanket' sets $B_\ib = \{ \ib' \ne \ib: |i_j - i_j'| \le 1, \forall j \in [d]\}$, and note that $|B_{\ib}| \le 3^d$, for all $\ib \in \overline \Ical$.
Hence, by \eqref{translation} and \eqref{high_ex0}, and the fact that $|\overline \Ical| = O( \prod_j T_j)$, we have
\[
A_1 := \sum_{\ib \in \overline \Ical} \sum_{\ib' \in B_{\ib}} \PP ( E_{\ib} ) \PP ( E_{\ib'} ) \le |\overline \Ical| (3^d) \PP(E_{\zero})^2 = O( \tprod_j T_j^{-1}) = o(1).
\]

Now take $\ib \in \overline\Ical$ and $\ib' \in B_{\ib}$.  
We have
\[
\PP ( E_{\ib} \cap E_{\ib'} ) = 2 \PP ( E_{\zero} ) - \PP ( E_{\ib} \cup E_{\ib'} ).
\]
We have \eqref{translation} and \eqref{high_ex0}, and as in \eqref{high_ex0}, except that $\lambda(R(\ib) \cup R(\ib')) = 2$ when $\ib' \ne \ib$, we also have
\[
\PP ( E_{\ib} \cup E_{\ib'} ) = \PP \big\{ \exists  \tb \in R(\ib) \cup R(\ib') : \Xi(\tb) > u \big\} \sim 2 e^{-\tau} \prod_{j=1}^d T_j^{-1} \sim 2 \PP ( E_{\zero} ).
\]
This implies that 
\[\PP ( E_{\ib} \cap E_{\ib'} ) = o\big( \tprod_j T_j^{-1} \big).\]
This holds uniformly over $\ib$ by translation invariance (translating the whole blanket set $B_\ib$) and also uniformly over $\ib'$ in the blanket because there are at most $3^d$ of these. 
Hence,
\[
A_2 := \sum_{\ib \in \overline \Ical} \sum_{\ib' \in B_{\ib}} \PP ( E_{\ib} \cap E_{\ib'} ) \le |\overline \Ical| (3^d) \ o\big( \tprod_j T_j^{-1} \big) = o(1).
\]

Finally, by \eqref{translation} and \eqref{high_ex0},
\[
M := \sum_{\ib \in \overline \Ical} \PP ( E_{\ib} ) = |\overline \Ical| \PP ( E_{\zero} ) \sim e^{-\tau} |\overline \Ical| \prod_{j=1}^d T_j^{-1} = e^{-\tau} \prod_{j=1}^d \frac{\lceil T_j \rceil}{T_j} \rightarrow e^{-\tau}.
\]
In our context, the Poisson approximation result stated in \citep[Th 1]{arratia1989two} implies that
\[
\Big| \PP\big(\cap_{\ib \in \overline \Ical} E_{\ib}^\comp\big) - e^{-M} \Big| \le A_1 + A_2, 
\]
from which we derive 
\[
\PP \left( \cap_{\ib \in \overline \Ical} E_{\ib}^\comp \right) \rightarrow e^{-e^{-\tau}}.
\]

In exactly the same way, we can also derive
\[
\PP \left( \cap_{\ib \in \underline \Ical} E_{\ib}^\comp \right) \rightarrow e^{-e^{-\tau}}.
\]

Because
\[
\PP \left( \cap_{\ib \in \overline \Ical} E_{\ib}^\comp \right) \le \PP \{ \exists \tb \in \Tcal : \Xi(\tb) \le u \} \le \PP \left( \cap_{\ib \in \underline \Ical} E_{\ib}^\comp \right)
\]
we conclude that 
\beq \label{oracle_Xi}
\PP \{ \exists \tb \in \Tcal : \Xi(\tb) \le u \} \rightarrow e^{-e^{-\tau}}.
\eeq
We can then express this result in terms of the behavior of $\tilde \tau$, defined in \eqref{eq:tilde_tau},
\[
\PP \left\{ \tilde \tau \le \tau \right\} \rightarrow e^{-e^{-\tau}} = 1 - \alpha,
\]
when \eqref{tau} holds.
This being true for all fixed $\tau$, by \lemref{approx_tau} with \lemref{Z_epsnet}, for $\hat \tau$ defined in \eqref{eq:hat_tau}, we have
\[
\PP \{ \hat \tau > \tau \} \sim \PP \left\{ \tilde \tau > \tau \right\} \to \alpha.
\]
We then invert this to get 
\[
\lim_{n \rightarrow \infty} \PP \Big\{ \tsup_{\tb \in \Tcal \cap \ZZ^{d}} \xi[R(\hb^\star, \hb^\star \circ \tb)] > u \Big\} = \alpha,
\]
which is what we needed to prove.

\subsubsection{Proof of Theorem \ref{thm:oracle_type2}}
\label{sec:proof_oracle_type2}
We keep the same notation used in \secref{proof_oracle_type1}.
While we worked under the null, we are now working under the alternative.
Redefine $\tb^\star$ such that $\hb^\star \circ \tb^\star$ is the true location of the rectangle of activation.
Let $\Tcal' = \Tcal \cap \times_j (\ZZ/h_j^{\star})$ and define 
\[
\Ucal_\eta = \{\tb \in \Tcal : \lambda(R(\tb) \cap R(\tb^\star)) \ge 1 - \eta \} \quad \textrm{ and } \quad \Ucal = \{\tb \in \Tcal: R(\tb) \cap R(\tb') \neq \emptyset, \textrm{ for some } \tb' \in \Ucal_\eta \}.
\]
Recall the definition of $\Upsilon$ in \eqref{Upsilon}.  In our present context, we can parameterize it by $\tb \in \Tcal$, and it satisfies
\[
\Upsilon(\tb) = \mu \lambda(R(\tb) \cap R(\tb^\star)) + \Xi(\tb).
\]
Throughout the following we assume that $\mu - v \to c \in \RR \cup \{- \infty, +\infty\}$, where $v = v_n$ is defined in \eqref{v_o}.  Recall the definition of the power and write it as a function of $c$,
\[
\beta(c) = \lim_{n \rightarrow \infty} \PP \Big\{ \tsup_{\tb \in \Tcal'} y[R((\hb^\star, \hb^\star \circ \tb))] > u \Big\}.
\]
Note that $\beta(c)$ is well defined by Slutsky's theorem and is clearly nondecreasing in $c$.
Hence, it suffices to consider the case where $c \in \RR$.
By \lemref{xi_bound}, Part 2, and the fact that $u \to \infty$,
\[
\PP\Big\{\tsup_{\tb \in \Ucal} \Xi(\tb) \ge u\Big\} = o(1).
\]
Thus, since
\[
\PP\Big\{\tsup_{\tb \in \Tcal} \Xi(\tb) \ge u\Big\} - \PP\Big\{\tsup_{\tb \in \Ucal} \Xi(\tb) \ge u\Big\} \le \PP\Big\{\tsup_{\tb \in \Tcal \backslash \Ucal} \Xi(\tb) \ge u\Big\} \le \PP\Big\{\tsup_{\tb \in \Tcal} \Xi(\tb) \ge u\Big\},
\]
we have
\beq \label{T-U}
\PP\Big\{\tsup_{\tb \in \Tcal \backslash \Ucal} \Xi(\tb) \ge u\Big\} \to \alpha,
\eeq
by \eqref{oracle_Xi}.
Hence,
\beqn
\PP \Big\{ \tsup_{\tb \in \Tcal} \Upsilon(\tb) > u \Big\} 
&\ge& \PP \Big\{ \tsup_{\tb \in \Tcal \backslash \Ucal} \Upsilon(\tb) > u \Big\} + \PP \{ \Upsilon(\tb^\star) > u\} \, \PP \Big\{ \tsup_{\tb \in \Tcal \backslash \Ucal} \Upsilon(\tb) \le u \Big\} \\
&\to& \alpha + \bar \Phi(c) (1 - \alpha). 
\eeqn
Select $\eta\rightarrow 0$ such that $\mu \eta \rightarrow \infty$.
By \lemref{xi_bound}, Part 2, we know that
\[
\tsup_{\tb \in \Ucal_\eta} \Upsilon(\tb) - \Upsilon(\tb^\star) \le \tsup_{\tb \in \Ucal_\eta} |\Xi(\tb) - \Xi(\tb^\star)| = O_\PP(1).
\]
Hence,
\beq \label{U_eta_conv}
\PP \Big\{ \tsup_{\tb \in \Ucal_\eta} \Upsilon(\tb) > u \Big\} \rightarrow \bar \Phi(c).
\eeq
Again by \lemref{xi_bound}, Part 2,
\[
\sup_{\tb \in \Ucal \backslash \Ucal_\eta} \Upsilon(\tb) \le \mu (1 - \eta) + O_\PP(1).
\]
Thus, in probability,
$
\mu - \sup_{\tb \in \Ucal \backslash \Ucal_\eta} \Upsilon(\tb) \rightarrow \infty,
$
implying that
\beq \label{U-U}
\PP \Big\{ \tsup_{\tb \in \Ucal \backslash \Ucal_\eta} \Upsilon(\tb) > u \Big\} \rightarrow 0.
\eeq
We then have
\beqn
\PP \Big\{ \tsup_{\tb \in \Tcal} \Upsilon(\tb) > u \Big\} 
&\le& \PP \Big\{ \tsup_{\tb \in \Tcal \backslash \Ucal} \Upsilon(\tb) > u \Big\} + \PP \Big\{ \tsup_{\tb \in \Ucal \backslash \Ucal_\eta} \Upsilon(\tb) > u \Big\} \\
&&\quad +\ \PP \Big\{ \tsup_{\tb \in \Ucal_\eta} \Upsilon(\tb) > u\Big\} \PP \Big\{ \tsup_{\tb \in \Tcal \backslash \Ucal} \Upsilon(\tb) \le u \Big\} \\
&\to& \alpha + \bar \Phi(c) (1 - \alpha), 
\eeqn
where the inequality is by independence of $(\Xi(\tb), \tb \in \Ucal_\eta)$ and $(\Xi(\tb), \tb \in \Tcal \setminus \Ucal)$, and the convergence is by \eqref{T-U}, \eqref{U_eta_conv}, and \eqref{U-U}.
We conclude that
\[
\beta(c) = \alpha + \bar \Phi(c) (1 - \alpha),
\]
and by \lemref{approx_tau} and \lemref{Z_epsnet}, we find that this holds for the discrete scan statistic as long as $\underline h = \omega(\log n)$, so that \eqref{ab_eps} is satisfied.

%

\subsubsection{Proof of Theorem \ref{thm:multiscale_type1}}

We now redefine $u$ as the critical value in \eqref{u_m}.
We assume that we are under the null.
Applying \citep[Th 1.4]{kabluchko2011extremes}, 
with $a \gets 1$ and $n \gets n / \underline h$, we get
\[
\lim_{n \rightarrow \infty} \PP \Big\{ \tsup_{\wb \in \Wcal} \Xi(\wb) \ge u \Big\} = \alpha.
\]
This translates into
\[
\lim_{n \rightarrow \infty} \PP \{ \tilde \tau > \tau \} = \alpha.
\]
Now we may apply \lemref{approx_tau} with \lemref{Z_epsnet} to obtain that the statistic $\hat \tau$ defined in \eqref{eq:hat_tau} satisfies $|\tilde \tau - \hat \tau| = o_\PP(1)$ and, therefore, we also have
\[
\lim_{n \rightarrow \infty} \PP \{ \hat \tau > \tau \} = \alpha.
\]
We then invert this to get
\[
\lim_{n \rightarrow \infty} \PP \Big\{ \tsup_{\wb \in \Wcal \cap \ZZ^{2d}} \xi[R(\wb)] \ge u \Big\} = \alpha.
\]

\subsubsection{Proof of Theorem \ref{thm:multiscale_type2}}
\label{sec:proof_multiscale_type2}
We assume that we are under the alternative.
The arguments are essentially identical to those in \secref{proof_oracle_type2}, except that this time both the scale and location vary.  In particular, we work with $\Wcal' = \Wcal \cap \ZZ^{2d}$, and 
\[
\Ucal_\eta = \Big\{\wb \in \Wcal : \frac{\lambda(R(\wb) \cap R(\wb^\star))}{\sqrt{\lambda(R(\wb))\lambda(R(\wb^\star))}} \ge 1 - \eta \Big\} \quad \textrm{ and } \quad \Ucal = \{\wb \in \Wcal: R(\wb) \cap R(\wb') \neq \emptyset, \forall \wb' \in \Ucal_\eta \},
\]
where $\wb^\star$ denotes the true scale and location of the rectangle of activation.  The remaining of the proof is now exactly the same.  

\subsubsection{Preliminaries}
The lemmata stated and proved in this section will be used to prove of Theorems~\ref{thm:adaptive_type1} and~\ref{thm:adaptive_type2}.
Until further notice, $u(\hb)$ (or $u_n(\hb)$ if we choose not to suppress the dependence on $n$) denotes the critical value defined in \eqref{u_a} while $v(\hb) = v_{n,\hb}$ denotes the function of $\hb$ in \eqref{v_a}.
The parameter $\tau$ remains fixed throughout.

The following technical lemma is used throughout this section.
\begin{lemma}
\label{lem:ua_lipschitz}
There exists $L > 0$ such that for all $\wb = (\hb,\tb), \wb' = (\hb',\tb') \in \Wcal$ such that $\delta(\wb,\wb') \le \epsilon_0$ as specified in \lemref{ada_lipschitz}, 
\[
|u(\hb) - u(\hb')| \le L \delta(\wb,\wb').
\]
\end{lemma}

\begin{proof}
Recall the notation introduced in \eqref{eq:hat_tau}, where we now abbreviate $a(\hb) = a((\hb,\tb))$ and $b(\hb) = b((\hb,\tb))$, and these functions are specified in \lemref{ada_lipschitz}.  We have
\[
|u(\hb) - u(\hb')| \le |a(\hb) - a(\hb')| + \tau \Big| \frac{1}{a(\hb)} - \frac{1}{a(\hb')} \Big| + \Big| \frac{b(\hb)}{a(\hb)} - \frac{b(\hb')}{a(\hb')} \Big|.
\]
By \lemref{ada_lipschitz}, $|a(\hb) - a(\hb') | \le L \delta(\wb',\wb)$ for some $L > 0$ and $\delta(\wb',\wb) \le \epsilon_0$.
Working with the second term, we obtain,
\[
\Big| \frac{1}{a(\hb)} - \frac{1}{a(\hb')} \Big| = \frac{|a(\hb) - a(\hb')|}{a(\hb)a(\hb')} \le L \frac{\delta(\wb,\wb')}{a(\hb)a(\hb')} \le L \delta(\wb,\wb'),
\]
using the fact that $a(\hb), a(\hb') \ge 1$ because $n/\underline h \ge n / \overline h \ge e$ by assumption.
Working with the third term,
\[
\begin{split}
\Big| \frac{b(\hb)}{a(\hb)} - \frac{b(\hb')}{a(\hb')} \Big| 
&\le \Big| \frac{b(\hb)}{a(\hb)} - \frac{b(\hb')}{a(\hb)} \Big| + \Big| \frac{b(\hb')}{a(\hb)} - \frac{b(\hb')}{a(\hb')} \Big| \\
&\le \frac{L \delta(\wb,\wb')}{a(\hb)} + \frac{b(\hb')}{a(\hb')} \Big| \frac{a(\hb) - a(\hb')}{a(\hb)} \Big| \le L (1 + C) \delta(\wb,\wb'),
\end{split}
\]
because there exists a $C$ such that $b(\hb)/a(\hb) \le C$ for all allowed $\hb$. 
Combining these we find that $u(\hb)$ is indeed (locally) Lipschitz with respect to $\delta$, with constant $L' = L (2 + \tau + C)$.
\end{proof}

We will introduce some notation for the following lemmata.
Let $\hb \in [e \underline h, \overline h]^d$, define $\Tcal(\hb) = \times_{j=1}^d h_j [\lceil n / h_j \rceil]$, and let $\tb \in \Tcal(\hb)$.
Define the set
\[
K_{(\hb,\tb)} = \Big\{ (\gb,\sbb) \in \Wcal : \gb/\hb \in [e^{-1},1]^d, \sbb \in  [\tb - \hb, \tb] \Big\}
\]
and the event 
\beq \label{Ew-def}
E_{(\hb,\tb)} = \Big\{ \exists (\gb,\sbb) \in K_{(\hb,\tb)} : \Xi(\gb,\sbb)> u(\gb) \Big\}.
\eeq

\begin{lemma}
\label{lem:nonconstant_scan}
Let $\kb \in \ZZ_+^d$ be fixed in $n$, define $\hb = \underline h e^{\kb}$, and let $\tb \in \Tcal(\hb)$.  We have
\[
\PP \big( E_{(\hb,\tb)} \big) \sim e^{-\tau} n^{-d} \prod_{j=1}^d h_j [k_j^{-1} - (1 + k_j)^{-1}].
\]
\end{lemma}

\begin{proof}
First, by location invariance 
\[
\PP \Big\{ \exists (\gb,\sbb) \in K_{(\hb,\tb)} : \Xi(\gb,\sbb) > u_n(\gb) \Big\} = \PP \Big\{ \exists (\gb,\sbb) \in K_{(\hb,\hb)} : \Xi(\gb,\sbb) > u_n(\gb) \Big\}.
\]
Also, because for $(\gb_0,\sbb_0),(\gb_1,\sbb_1) \in \Wcal$,
\[
\frac{\lambda(R(\underline h^{-1} \gb_0,\underline h^{-1} \sbb_0) \cap R(\underline h^{-1} \gb_1,\underline h^{-1} \sbb_1))}{\sqrt{\lambda(R(\underline h^{-1} \gb_0, \underline h^{-1} \sbb_0))\lambda(R(\underline h^{-1} \gb_1,\ab \underline h^{-1} \sbb_1))}} = \frac{\lambda(R(\gb_0,\sbb_0) \cap R(\gb_1,\sbb_1))}{\sqrt{\lambda(R(\gb_0,\sbb_0))\lambda(R(\gb_1,\sbb_1))}},
\]
rescaling the set $\Wcal$ by $\underline h^{-1}$ does not change the covariance structure of $\Xi[\bar \Rcal]$.
Thus,
\beq \label{KwK0}
\PP \Big\{ \exists (\gb,\sbb) \in K_{(\hb,\hb)} : \Xi(\gb,\sbb) \ge u_n(\gb) \Big\} = \PP \Big\{ \exists (\gb,\sbb) \in K_0 : \Xi(\gb,\sbb) \ge u_n(\underline h \gb) \Big\}.
\eeq
for 
\[
K_0 = \underline h^{-1} K_{(\hb,\hb)} = [e^{\kb - \one},e^{\kb}] \times [\zero, e^{\kb}].
\]
Let $\Lambda(\hb) = 4^{-d} \prod_{i=1}^d h_i^{-2}$ be the high excursion intensity from Lemma \ref{lem:excursion}.
Set $c = 4^d \sqrt{2 \pi}$.
We now check the conditions of \lemref{chan_nonconstant}.
Here the boundary functions are $u_n(\hb)$ defined in \eqref{u_a}, and satisfy \eqref{chan_nonconstant1} with $\zeta = (\log n)^{-3/2}$.
To see this, first, notice that $K_0$ is fixed and Jordan measurable.
By \eqref{eq:delta_bound}, for any fixed $\gamma > 0$ small enough and $\wb_0 = (\gb_0,\sbb_0), \wb_1 = (\gb_1,\sbb_1) \in [K_0]_{\gamma}$ such that $\|\wb_0 - \wb_1\|_\infty \le \zeta$,
\[
\delta^2(\underline h \wb_0,\underline h \wb_1) = \delta^2(\wb_0,\wb_1) \le 2 \sum_j \frac{|g_{0,j} - g_{1,j}| + |s_{0,j} - s_{1,j}|}{\sqrt{g_{0,j} g_{1,j}}} \le 4 \zeta C_\gamma \sum_j e^{1 - k_j}
\]
where $C_\gamma$ is a small constant.
Thus by \lemref{ua_lipschitz}, for $\zeta$ small enough, 
\[
|u_n(\underline h \gb_0) - u_n(\underline h \gb_1)| \le L \delta(\underline h \wb_0,\underline h \wb_1) \le L \sqrt{4 \zeta C_\gamma \sum_j e^{1 - k_j}}.
\]
We also have that 
\beq \label{sup_u}
\sup_{(\gb,\sbb) \in [K_0]_{2\zeta}} u_n(\gb) = O(\sqrt{\log n})
\eeq 
and so
\beq \label{u2diff}
|u_n(\underline h \gb_0)^2 - u_n(\underline h \gb_1)^2| = O( \sqrt{\zeta \log n}) = o(1)
\eeq
uniformly over such $\wb_0,\wb_1$, which verifies \eqref{chan_nonconstant2}.
Furthermore, recalling that in the notation of \lemref{chan_nonconstant}, we have $\alpha = 1$ and $p = 2d$, we finally get \eqref{chan_nonconstant1}
\[
\sup_{(\gb,\sbb) \in [K_0]_{\zeta}} u_n(\gb)^{-4d} = O((\log n)^{-2d}) = o(\zeta).
\]
Hence, we have established the conditions of \lemref{chan_nonconstant}.
Applying it we have 
\begin{eqnarray*}
&&\PP \Big\{ \exists (\gb,\sbb) \in K_0 : \Xi(\gb,\sbb) > u_n(\underline h \gb) \Big\}\\
& \sim &\int_{[\zero,e^\kb]} \int_{[e^{\kb - \one},e^\kb]} \psi(u_n(\underline h \gb)) \Lambda(\gb) u_n(\underline h \gb)^{4d} {\rm d} \gb {\rm d} \tb \\
&= &\frac{e^{\sum_j k_j}}{\sqrt{2\pi}} \int_{[e^{\kb - \one},e^\kb]} \Lambda(\gb) u_n(\underline h \gb)^{4d - 1} e^{-u_n(\underline h \gb)^2 / 2} {\rm d} \gb\\
&= &\frac{e^{\sum_j k_j}}{4^d \sqrt{2\pi}} \int_{[e^{\kb - \one},e^\kb]} \big(\prod_j g_j^{-2}\big) u_n(\underline h \gb)^{4d - 1} e^{-u_n(\underline h \gb)^2 / 2} {\rm d} \gb.\\
\end{eqnarray*}
We have from Lemma \ref{lem:u_rate} that 
\[
u_n(\underline h \gb)^{4d - 1} e^{-u_n(\underline h \gb)^2 / 2} \Big/ \Big[ \frac{c}{e^\tau \prod_j (n \log^2 (e g_j) / (\underline h g_j))} \Big] \rightarrow 1, \quad \forall \gb \in [e^{\kb - \one},e^\kb],
\]
which implies that 
\[
\frac{n^d}{\underline h^d} u_n(\underline h \gb)^{4d - 1} e^{-u_n(\underline h \gb)^2 / 2}  \rightarrow c e^{-\tau} \prod_j \frac{g_j}{\log^2(e g_j)}  , \quad \forall \gb \in [e^{\kb - \one},e^\kb].
\]
We see by \lemref{u_rate}, uniformly over $n$,
\[
\sup_{(\gb,\sbb) \in K_0} \frac{n^d}{\underline h^d} u_n(\underline h \gb)^{4d - 1} e^{-u_n(\underline h \gb)^2 / 2} < + \infty.
\]
By dominated convergence,
\[
\frac{n^d}{\underline h^d} \int_{[e^{\kb - \one},e^\kb]}  \prod_j g_j^{-2} (u_n(\underline h \gb))^{4d - 1} e^{-u_n(\underline h \gb)^2 / 2} {\rm d} \gb \rightarrow \int_{[e^{\kb - \one},e^\kb]} c e^{-\tau} \prod_j [g_j \log^2(e g_j)]^{-1} {\rm d} \gb.
\]
Thus,
\[
\frac{e^{\sum_j k_j}}{4^d \sqrt{2\pi}} \int_{[e^{\kb - \one},e^\kb]} \big(\prod_j g_j^{-2}\big) u_n(\underline h \gb)^{4d - 1} e^{-u_n(\underline h \gb)^2 / 2} {\rm d} \gb \sim 
\frac{\underline h^d e^{\sum_j k_j}}{n^d e^\tau} \prod_j [k_j^{-1} - (1 + k_j)^{-1}].
\]
We have our result because $h_j = \underline h e^{k_j}, j \in [d]$.
\end{proof}

\begin{lemma}
\label{lem:nonconstant_bound}
Resume the notation of \lemref{nonconstant_scan}.
There exists a constant $C > 0$ not depending on $n$, $\kb$, or $\tb$ (but possibly dependent on $\tau$ or $d$) such that
\[
\PP \big( E_{(\hb,\tb)} \big) \le C \prod_j \frac{h_j}{n k_j^2}.
\]
\end{lemma}

\begin{proof}
Let $\wb = (\hb, \tb)$.  By \eqref{KwK0},
\[
\PP(E_\wb) = \PP \Big\{ \exists (\gb,\sbb) \in K_0 : \Xi(\gb,\sbb) > u(\underline h \gb) \Big\} \le \PP \Big\{ \exists (\gb,\sbb) \in K_0 : \Xi(\gb,\sbb) > c_{\wb} \Big\}
\]
where
\[
c_{\wb} = \min_{(\gb,\sbb) \in K_{\wb}} u(\gb) = \min_{\gb \in [e^{-1} \hb, \hb]} u(\gb).
\]
By scale invariance, 
\[\begin{split}
\PP \Big\{ \exists (\gb,\sbb) \in K_0 : \Xi(\gb,\sbb) \ge c_{\wb} \Big\}
&= \PP \Big\{ \exists \gb \in [e^{-1},1]^d, \sbb \in [0,1]^d : \Xi(\gb,\sbb) \ge c_{\wb} \Big\} \\
&\le C_1 c_\wb^{4d} \psi(c_\wb) \int_{\wb' \in [e^{-1},1]^d \times [0,1]^d} \Lambda(\wb') {\rm d} \wb',
\end{split}\]
for some constant $C_1 > 0$, by an application of \lemref{chan_constant}.
On the one hand, using the form for $\Lambda$ given in \lemref{excursion}, we get
\[
\int_{\wb' \in [e^{-1},1]^d \times [0,1]^d} \Lambda(\wb') {\rm d} \wb' = 4^{-d} \prod_j \int_{e^{-1}}^1 {g_j}^{-2} {\rm d}g_j < \infty.
\]
On the other hand, by \lemref{u_rate} there is a constant $C_2$ (not dependent on $\kb$, $n$, or $\tb$) such that $\forall (\gb,\sbb) \in K_0$,
\[
u(\underline h \gb)^{4d} \psi(u(\underline h \gb)) \le e^{-\tau - \kappa} e^{-v(\underline h \gb)^2/2} \Big(1 + C_2 \frac{\log v(\underline h \gb)}{v(\underline h \gb)}\Big).
\]
Since $\min_{\gb \in [e^{\kb - \one}, e^\kb]} v(\underline h \gb) \to \infty$ uniformly over $\kb \in \ZZ_+^d$, we have 
\[
c_\wb^{4d} \psi(c_\wb) \le (1+o(1)) e^{-\tau-\kappa} \exp\Big[- \tfrac12 \min_{\gb \in [e^{\kb - \one}, e^\kb ]} v^2(\underline h \gb)\Big] \le C_3 \prod_j \frac{h_j}{n} \log^{-2}(h_j/\underline h),
\]
where $C_3$ is some constant, using the fact that $\min_{\hb' \in [e^{-1} \hb,\hb]} v(\hb') = v(\hb)$ for $h_j \ge e \underline h, \forall j$.
We conclude that there exists a constant $C_4$ such that, for all such $\wb$,
\[
\PP(E_\wb) \le C_4 \prod_j \frac{h_j}{n} \log^{-2}(h_j/\underline h).
\]
\end{proof}

\begin{lemma}
\label{lem:low_scale_adascan}
For all $A \in \ZZ_+$, let $\Ucal_A = \big\{ (\hb, \tb) \in \Wcal : h_j \le \underline h e^A, \forall j \in [d] \big\}$.
Then 
\[
\lim_{A \rightarrow \infty} \lim_{n \rightarrow \infty} \PP \Big\{ \Xi(\hb,\tb) > u_n(\hb) \textrm{ for some } (\hb,\tb) \in \Ucal_A \Big\} = \alpha.
\]
\end{lemma}

\begin{proof}
Resume the notation and definitions of \lemref{nonconstant_bound}.
We partition the space $\Wcal$ into blocks in the scale and location parameters.
Define 
\[
\overline \Ical = \Big\{ (\hb,\tb) \in \Wcal : \exists \kb \in [A]^d, \hb = \underline h e^\kb, \tb \in \Tcal(\hb) \Big\}
\]
and $\underline \Ical = \overline \Ical \cap \Ucal_A$
whereby $\cup_{\wb \in \underline \Ical} K_\wb \subseteq \Ucal_A \subseteq \cup_{\wb \in \overline \Ical} K_\wb$.
Recall that for $\kb \in [A]^d$ and $\hb = \underline h e^{\kb}$,
\[
\PP \big( E_{(\hb,\tb)}\big) = \PP \big( E_{(\hb,\hb)}\big), \quad \forall \tb \in \Tcal(\hb)
\]
by translation invariance.
By \lemref{nonconstant_scan},
\[
\PP \big( E_{(\hb,\hb)}\big) \sim |\Tcal(\hb)|^{-1} e^{-\tau} \prod_j \big[k_j^{-1} - (1+ k_j)^{-1} \big].
\]
We partition the set $\Ucal_A$ into the blocks $\{ K_\wb : \wb \in \overline \Ical \}$ and then use the Chen-Stein Poisson approximation to derive $\PP(\cup_{\wb \in \overline \Ical} E_\wb)$.
We have
\[\begin{split}
M := \sum_{\wb \in \overline \Ical} \PP ( E_\wb) 
&= \sum_{\kb \in [A]^d} \ \sum_{\tb \in \Tcal(\underline h e^\kb)} \PP \big( E_{(\underline h e^\kb,\tb)}\big) = \sum_{\kb \in [A]^d} \ |\Tcal(\underline h e^\kb)| \PP \big( E_{(\underline h e^\kb,\underline h e^\kb)}\big) \\
&\rightarrow e^{-\tau} \sum_{\kb \in [A]^d} \prod_j \big[ k_j^{-1} - (1 + k_j)^{-1} \big].
\end{split}\] 
We then have that
\[
\sum_{\kb \in [A]^d} \prod_j \big[ k_j^{-1} - (1 + k_j)^{-1} \big] = \Big( \tsum_{k = 1}^A [k^{-1} - (1 + k)^{-1}] \Big)^d = \Big( 1 - 1/(1 + A) \Big)^d.
\]
Thus, we obtain that 
\[\lim_{A \rightarrow \infty} \lim_{n \rightarrow \infty} M \rightarrow e^{-\tau}.\]

Two events, $E_{(\hb,\tb)}, E_{(\gb,\sbb)}$, are independent if $|t_j - s_j| > 2 (h_j \vee g_j)$, for some $j \in [d]$.
Consider then the `blanket' sets 
\[B_{(\hb,\tb)} = \Big\{ (\gb,\sbb) \in \overline \Ical \setminus \{(\hb, \tb)\} : \forall j \in [d], |t_j - s_j| \le 2 (h_j \vee g_j) \Big\}.\]  
We have
\[
|B_{(\hb,\tb)}| \le \sum_{\kb \in [A]^d} \Big|\Big\{ \sbb \in \Tcal(\underline h e^\kb) : \exists j \in [d], |t_j -  s_j| \le 2 \underline h e^A  \Big\}\Big| \le 4^d \sum_{\kb \in [A]^d} \left\lceil e^{\sum_{j=1}^d (A - k_j)} \right\rceil.
\]
which is a constant depending only on $d$ and $A$.  Thus,
\[
A_1 := \sum_{\wb \in \overline \Ical} \sum_{\wb' \in B_\wb} \PP (E_\wb) \PP(E_{\wb'}) \le \Big( \max_{\wb \in \overline \Ical} |B_\wb| \PP(E_\wb) \Big) \sum_{\wb \in \overline \Ical} \PP (E_\wb) = o\Big( \sum_{\wb \in \overline \Ical} \PP (E_\wb) \Big),
\]
since $\max_{\wb \in \overline \Ical} \PP(E_\wb) = o(1)$ by \lemref{nonconstant_bound}.
Take $\wb \in \overline \Ical$ and $\wb' = (\hb',\tb') \in B_{\wb}$. We have
\[
\PP ( E_{\wb} \cap E_{\wb'} ) = \PP (E_\wb) + \PP(E_{\wb'}) - \PP (E_\wb \cup E_{\wb'}). 
\]
By same exact arguments underlying the proof of \lemref{nonconstant_scan},
\[
\PP (E_\wb \cup E_{\wb'}) = \PP \Big\{ \exists (\hb_0,\tb_0) \in K_\wb \cup K_{\wb'} : \Xi(\hb_0,\tb_0) > u(\hb_0) \Big\} \sim \PP (E_{\wb}) + \PP (E_{\wb'}).
\]
We can also see from \lemref{nonconstant_bound} that, uniformly over $\wb' \in B_{\wb}$,
\[
\PP (E_{\wb'}) = O( \PP (E_{\wb}) ).
\]
Again by translation invariance and the fact that both $|[A]^d|$ and $|B_\wb|$ are bounded in $n$,
\begin{eqnarray*}
A_2:&= &\sum_{\kb \in [A]^d} \sum_{\tb \in \Tcal(\underline h e^\kb)} \sum_{\wb' \in B_{(\underline h e^\kb, \tb)}} \PP ( E_{(\underline h e^\kb,\tb)} \cap E_{\wb'} ) \\
&=& \sum_{\kb \in [A]^d} |\Tcal(\underline h e^\kb)|  \sum_{\wb' \in B_{(\underline h e^\kb, \underline h e^\kb)}} o\big[\PP ( E_{(\underline h e^\kb, \underline h e^\kb)}) + \PP (E_{\wb'})\big] \\
&=& o(M) = o(1).
\end{eqnarray*}
This shows that the events $E_\wb$ over $\overline \Ical$ have finite-range dependence.
Hence, by \citep[Th 1]{arratia1989two} we have that
\[
\Big| \PP \left( \cap_{\wb \in \overline \Ical} E_{\wb}^\comp \right) - e^{-M} \Big| \le A_1 + A_2 = o(1).
\]

This also holds with $\underline \Ical$ in place of $\overline \Ical$, and with $\lim_{n \rightarrow \infty}M$ unaffected.  So the proof is complete.
\end{proof}

\begin{lemma}
\label{lem:up_scale_adascan}
With $\Ucal_A$ defined in \lemref{low_scale_adascan}, we also have
\[
\lim_{A \rightarrow \infty} \lim_{n \rightarrow \infty} \PP \big\{ \exists (\hb,\tb) \in \Wcal \setminus \Ucal_A : \Xi(\hb, \tb) > u(\hb) \big\} = 0.
\]
\end{lemma}

\begin{proof}
We keep the same notation as in the previous proof.
Define the event
\[
\Ecal_A = \Big\{ \exists (\hb, \tb) \in \Wcal \backslash \Ucal_A : \Xi(\hb, \tb) > u(\hb) \Big\}.
\]
Note that $\Ecal_A$ depends on $n$ via $u(\hb)$ in \eqref{u_a}.
By the union bound,
\beq \label{pnk}
\PP (\Ecal_A) \le \sum_{\kb \in [\log n]^d \backslash[A]^d} p_{n,\kb}, \quad \text{where } p_{n,\kb} := \sum_{\tb \in \Tcal(\underline h e^\kb)} \PP ( E_{(\underline h e^\kb,\tb)}) = \big(\tprod_j \lceil n / \underline h e^{k_j} \rceil\big) \PP ( E_{(\underline h e^\kb, \underline h e^\kb)}),
\eeq
by translation invariance. 
By \lemref{nonconstant_scan},
\begin{eqnarray*}
\sum_{\kb \in \ZZ_+^d \backslash[A]^d} \lim_{n \rightarrow \infty} p_{n,\kb} &=& e^{-\tau} \sum_{\kb \in \ZZ_+^d \backslash[A]^d} \prod_{j=1}^d \Big[ k_j^{-1} - (1 + k_j)^{-1} \Big]\\
&=& e^{-\tau} \Big( \sum_{\kb \in \ZZ_+^d} \tprod_j \big[ k_j^{-1} - (1 + k_j)^{-1} \big] - \sum_{\kb \in [A]^d} \tprod_j \big[ k_j^{-1} - (1 + k_j)^{-1} \big] \Big)\\
&=& e^{-\tau} \Big(1  - \big(1 - 1/(1+A)\big)^d \Big).
\end{eqnarray*}
By \lemref{nonconstant_bound}, 
\[
p_{n,\kb} \le \Big(\prod_j \lceil n/h_j\rceil \Big) C \prod_j \frac{\underline h e^{k_j}}{n k_j^2} \le C_1 \prod_j k_j^{-2}
\]
for $C_1 > 0$ not dependent on $n$ or $\kb$.
Hence, $D_\kb := C_1 \prod_j k_j^{-2}$ is a dominating sequence that is independent of $n$ and summable over $\ZZ_+^d$, and satisfies $p_{n,\kb} \le D_\kb$. 
Thus, we can apply dominated convergence and conclude that
\[
\lim_{n \rightarrow \infty} \PP(\Ecal_{n,A}) \le \lim_{n \rightarrow \infty} \sum_{\kb \in \ZZ_+^d \backslash[A]^d} p_{n,\kb} = \sum_{\kb \in \ZZ_+^d \backslash[A]^d} \lim_{n \rightarrow \infty}  p_{n,\kb} =  e^{-\tau} \Big(1  - \big(1 - 1/(1+A)\big)^d \Big).
\]
Since the RHS tends to zero as $A \to \infty$, the proof is complete.
\end{proof}

\subsubsection{Proof of Theorem \ref{thm:adaptive_type1}}
By \lemref{low_scale_adascan} and \lemref{up_scale_adascan}, 
\begin{eqnarray}
\lim_{n \rightarrow \infty} \PP \Big\{ \exists (\hb,\tb) \in \Wcal : \Xi(\hb, \tb) > u(\hb) \Big\} 
&=& \lim_{A \rightarrow \infty} \Big[ \lim_{n \rightarrow \infty} \big( \PP \Big\{ \exists (\hb,\tb) \in \Ucal_A : \Xi(\hb,\tb) > u(\hb) \Big\} \notag \\
&&\quad +\ \lim_{n \rightarrow \infty} \PP \Big\{ \exists (\hb,\tb) \in \Wcal \backslash \Ucal_A : \Xi(\hb,\tb) > u(\hb) \Big\} \Big] \notag \\
&=& \alpha. \label{W_alpha}
\end{eqnarray}
Hence, the random variable $\tilde \tau$ defined in \eqref{eq:tilde_tau} satisfies
\[
\lim_{n \rightarrow \infty} \PP \{ \tilde \tau > \tau \} = \lim_{n \rightarrow \infty} \PP \Big\{ \exists (\hb,\tb) \in \Wcal : \Xi(\hb, \tb) > u(\hb) \Big\} = \alpha.
\]
Now, we may apply \lemref{approx_tau} with \lemref{Z_epsnet} to obtain that the statistic $\hat \tau$ defined in \eqref{eq:hat_tau} satisfies $|\tilde \tau - \hat \tau| = o_\PP(1)$ and, therefore,
\[
\lim_{n \rightarrow \infty} \PP \Big\{ \exists (\hb,\tb) \in \Wcal \cap \ZZ^{2d} : \xi[R(\hb, \tb)] > u(\hb) \Big\} = \lim_{n \rightarrow \infty} \PP \{ \hat \tau > \tau \} = \alpha.
\]

\subsubsection{Proof of Theorem \ref{thm:adaptive_type2}}
We resume the notation introduced in Sections~\ref{sec:proof_oracle_type2} and~\ref{sec:proof_multiscale_type2}.  The arguments here are very similar, so that we will omit some details.
We focus on the case in which $\mu - v(\hb^\star) \rightarrow c$.
By \lemref{xi_bound}, Part 2, and the fact that $\min_{\hb} v(\hb) \to \infty$,
\[
\PP\big\{\exists (\hb,\tb) \in \Ucal : \Xi(\hb,\tb) > u(\hb)\big\} \le \PP\big\{\exists (\hb,\tb) \in \Ucal : \Xi(\hb,\tb) > v(\hb) - o(1) \big\} = o_\PP(1).
\]
Thus, combining this with \eqref{W_alpha}, we have
\[
\begin{aligned}
\PP\big\{\exists (\hb,\tb) \in \Wcal \backslash \Ucal :  \Xi(\hb,\tb) > u(\hb)\big\} \rightarrow \alpha.
\end{aligned}
\]
Hence,
\[
\begin{aligned}
\PP \big\{\exists (\hb,\tb) \in \Wcal :  \Upsilon(\hb,\tb) > u(\hb) \big\} 
&\ge \PP \big\{\exists (\hb,\tb) \in \Wcal \backslash \Ucal :\Upsilon(\hb,\tb) > u(\hb) \big\} \\
&\quad + \PP \big\{ \Upsilon(\wb^\star) > u(\hb^\star)\big\} \PP \big\{ \exists (\hb,\tb) \in \Wcal \backslash \Ucal : \Upsilon(\hb,\tb) \le u(\hb) \big\}\\
&\rightarrow \alpha + \bar \Phi(c) (1 - \alpha). 
\end{aligned}
\]
By \lemref{ua_lipschitz}, there is some $L > 0$ such that for all $\wb = (\hb,\tb) \in \Ucal_\eta$, $u(\hb^\star) - u(\hb) \le L \delta(\wb^\star,\wb)$ for $\delta(\wb^\star,\wb) \le \epsilon_0$.
Select $\eta\rightarrow 0$ such that $\mu \eta \rightarrow \infty$.
For $\wb = (\hb,\tb) \in \Ucal_\eta$,
\[
\begin{aligned}
( \Upsilon(\wb) - u(\hb) ) - ( \Upsilon(\wb^\star) - u(\hb^\star) ) = [\Xi(\wb) - \Xi(\wb^\star)] + [m(\wb) - m(\wb^\star)] + [u(\hb^\star) - u(\hb)] \\
\le |\Xi(\wb) - \Xi(\wb^\star)| + L \delta(\wb,\wb^\star).
\end{aligned}
\]
By \lemref{xi_bound}, Part 2, 
\[
\sup_{\wb \in \Ucal_\eta} |\Xi(\wb) - \Xi(\wb^\star)| = O_\PP(1).
\]
By this, the fact that if $\eta \rightarrow 0$ then $\sup_{\wb \in \Ucal_\eta}\delta(\wb,\wb^\star) \rightarrow 0$, and that $m(\wb^\star) \ge m(\wb)$,
\[
\sup_{\wb \in \Ucal_\eta} [\Upsilon(\wb) - u(\hb)] - [\Upsilon(\wb^\star) - u(\hb^\star)] = O_\PP(1).
\]
Hence,
\[
\PP \big\{\exists (\hb,\tb) \in \Ucal_\eta : \Upsilon(\hb,\tb) > u(\hb) \big\} \rightarrow \bar \Phi(c).
\]
Again by \lemref{xi_bound}, Part 2,
\[
\sup_{\wb \in \Ucal \backslash \Ucal_\eta} \Upsilon(\wb) \le \mu (1 - \eta) + O_\PP(1).
\]
Thus, 
\[
\mu - \sup_{\wb \in \Ucal \backslash \Ucal_\eta} \Upsilon(\wb) \rightarrow \infty.
\]
Hence,
\[
\PP \big\{\exists (\hb,\tb) \in \Ucal \setminus \Ucal_\eta : \Upsilon(\hb,\tb) > u(\hb) \big\} \rightarrow 0.
\]
The probability of exceedance can be bounded by
\[
\begin{aligned}
&\PP \big\{\exists (\hb,\tb) \in \Wcal : \Upsilon(\hb,\tb) > u(\hb) \big\} \\
&\quad \le \PP \big\{ \exists (\hb,\tb) \in \Wcal \setminus \Ucal : \Upsilon(\hb,\tb) > u(\hb) \big\} \\
&\quad\quad + \PP \big\{\exists (\hb,\tb) \in \Ucal \setminus \Ucal_\eta : \Upsilon(\hb,\tb) > u(\hb) \big\} \\
&\quad\quad\quad + \PP \big\{ \exists (\hb,\tb) \in \Ucal_\eta : \Upsilon(\hb,\tb) > u(\hb)\big\} \PP \big\{ \nexists (\hb,\tb) \in \Wcal \backslash \Ucal : \Upsilon(\hb,\tb) > u(\hb)\big\} \\
&\quad \rightarrow \alpha + \bar \Phi(c) (1 - \alpha), 
\end{aligned}
\]
by independence of $\{\Upsilon(\wb) : \wb \in \Ucal_\eta\}$ and $\{\Upsilon(\wb) : \wb \in \Wcal \setminus \Ucal\}$.

We conclude that
\[
\PP \big\{\exists (\hb,\tb) \in \Wcal :  \Upsilon(\hb,\tb) > u(\hb) \big\} 
\rightarrow \alpha + \bar \Phi(c) (1 - \alpha). 
\]
By \lemref{approx_tau} and \lemref{Z_epsnet}, we find that this also holds when $\Wcal$ is replaced by $\Wcal'$, as long as $\underline h = \omega(\log n)$.  And from this we conclude as in \secref{proof_oracle_type2}.

\subsubsection{Proof of \thmref{epsilon}}
For adaptive multiscale scan, \lemref{ada_lipschitz} allows us to apply the conclusion of \lemref{approx_tau} to the critical value \eqref{u_a}. For the multiscale scan statistic, \lemref{approx_tau} applies to the constant critical value \eqref{u_m}.
Let $\hat \tau$ be the result of the scan over the discrete set, $\Wcal \cap \ZZ^{2d}$, for either the (resp.~adaptive) multiscale scan, and let $\hat \tau_\epsilon$ be the scan over the $\epsilon$-covering.
Then by \lemref{Z_epsnet}, $\Wcal \cap \ZZ^{2d}$ is an $\epsilon'$-covering of $\Wcal$ for $\epsilon' = \sqrt{4d / \underline h} = o((\log n)^{-1/2})$.
Thus, we may apply \lemref{approx_tau}, unless $\mu = \omega(\sqrt{\log (n / \underline h)})$ under $H_1$, to show that $|\tilde \tau - \hat \tau| = o_\PP(1)$.
Likewise, when $\epsilon = o((\log n)^{-1/2})$ then $\hat \tau_\epsilon$ fulfills the conditions of \lemref{approx_tau} unless $\mu = \omega(\sqrt{\log (n / \underline h)})$ under $H_1$.
But $\mu = \omega(\sqrt{\log (n / \underline h)})$ implies that $\hat \tau, \hat \tau_\epsilon, \tilde \tau \rightarrow \infty$ because then $y[R^\star] = \omega_\PP(\sqrt{\log (n / \underline h)})$.
In this case, $\hat \alpha, \hat \alpha_\epsilon \rightarrow 0$.
When this is not the case then $|\hat \tau - \tilde \tau| = o_\PP(1)$ and $|\hat \tau_\epsilon - \tilde \tau| = o_\PP(1)$ by \lemref{approx_tau}, and so $|\hat \tau - \hat \tau_\epsilon| = o_\PP(1)$.
Because $\hat \alpha = 1 - \exp(-\exp(-\hat \tau))$ and $\hat \alpha_\epsilon = 1 - \exp(-\exp(-\hat \tau_\epsilon))$, the result follows by the continuous mapping theorem.

\subsubsection{Proof of Proposition \ref{prop:eps-cover}}
We now show that $\Rcal_\epsilon$ is an $\epsilon$-covering of $\Rcal$.
Specifically, for each $(\hb,\tb) \in \Wcal$, we construct $(\gb,\sbb)$ such that $R(\gb,\sbb) \in \Rcal_\epsilon$ and $\delta(R(\hb,\tb), R(\gb,\sbb)) \le \eps$.
Take $a_j = \lfloor \log_2 \frac{h_j \epsilon^2}{4d} \rfloor \ge \underline a$, for each $j \in [d]$.
Define
\[
g_j = \argmin\{ |h - h_j| : h \in 2^{a_j} \ZZ_+\}, \quad
s_j = \argmin\{ |s - t_j| : s \in 2^{a_j} \ZZ_+\}.
\]
We know that 
\[
\frac{4d}{h_j \epsilon^2} = 2^{- \log_2 \frac{h_j \epsilon^2}{4d}} \le 2^{-a_j} \le 2^{- \log_2 \frac{h_j \epsilon}{4d} + 1} = \frac{8d}{h_j \epsilon^2}.
\]
By the construction,
\[
2^{-a_j} |g_j - h_j| \le \frac 12.
\]
Hence, we have that 
\beq \label{g_int}
2^{-a_j} g_j \in \left[ 2^{-a_j} h_j - \frac 12 , 2^{-a_j} h_j + \frac 12 \right] \subseteq \left[ \frac{4d}{\epsilon^2} - \frac 12, \frac{8d}{ \epsilon^2} + \frac 12 \right].
\eeq
But because $2^{-a_j} g_j$ is constructed to be in $\ZZ_+$ then we know that it lies within $[\lceil 8 d / \epsilon^2 \rceil]$.
Therefore, $R(\gb,\sbb) \in \Rcal_\epsilon$.
Let $\epsilon^2 < 4d$.
It remains to show that $\theta((h_j,t_j),(g_j,s_j)) \le \epsilon^2 / 2d$ for all $j$, so that $\delta (\hb,\tb),(\gb,\sbb)) \le \epsilon$ by \eqref{eq:delta_bound}.  
We can see that
\[
|g_j - h_j| , |s_j - t_j| \le 2^{a_j - 1}, \textrm{ and } h_j \in 2^{a_j} \left[ \frac{4d}{\epsilon^2}, \frac{8d}{\epsilon^2} \right].
\]
Because $4d / \epsilon^2 \ge 1$ then $h_j \ge 2^{a_j}$.
Furthermore, $h_j^2 - 2^{a_j - 1} h_j$ is an increasing function for $h_j \ge 2^{a_j}$.
Hence,
\[
g_j h_j = h_j^2 - (h_j - g_j) h_j \ge h_j^2 - 2^{a_j - 1} h_j \ge 2^{2a_j} \big( \tfrac{16d^2}{\epsilon^4} - \tfrac{2d}{\epsilon^2} \big).
\]
We then have 
\[
\theta((h_j,t_j),(g_j,s_j)) \le \frac{|g_j - h_j| + |s_j - t_j|}{\sqrt{g_j h_j}} \le \left(\frac{16d^2}{\epsilon^4} - \frac{2d}{\epsilon^2} \right)^{-1/2} \le \frac{\epsilon^2}{2d} \big( 4 -  \tfrac{\epsilon^2}{2d}\big)^{-\frac 12} < \frac{\epsilon^2}{2d} ,
\]
since $\epsilon^2 < 4d$.

\subsubsection{Proof of Proposition \ref{prop:alg_eps_net}}
First, we establish that, for $\ab \in \{ \underline a, \ldots, \overline a \}^d$,
\begin{equation}
\label{eq:dyad_decomp}
(\textrm{dyad}_\ab \ast b_{\fb})(\tb) = (y \ast b_{2^\ab \circ \fb})(2^\ab \circ \tb), \quad \tb \in [n / 2^{\ab}].
\end{equation}
An induction on $\|\ab\|_1$, based on the recursion in Line~\ref{line:dyad}, 
gives
\[
\textrm{dyad}_\ab(\tb) = \sum_{\ib \in [2^{\ab}]} y(2^\ab \circ \tb + \ib), \quad \forall \tb \in \times_j [n / 2^{a_j}].
\]
Based on this, we have
\[\begin{split}
(\textrm{dyad}_\ab \ast b_\fb)(\tb) 
&= \sum_{\ib \in [\fb]} \textrm{dyad}_\ab (\ib + \tb) = \sum_{\ib \in [\fb]} \sum_{\kb \in [2^\ab]} y(2^\ab \circ (\ib + \tb) + \kb) \\
&= \sum_{\ib \in [2^{\ab} \fb]} y(\ib + 2^\ab \circ \tb) = (y \ast b_{2^\ab \circ \fb}) (2^\ab \circ \tb).
\end{split}\]
With \eqref{eq:dyad_decomp}, we can see that the statistic $\hat s$ in \algref{eps_rect} is equivalently expressed as
\[
\max_{\tb \in [n / 2^{\ab}]} y\big[ [2^{\ab}, 2^{\ab} (\tb + \fb) ] \big],
\]
confirming that \algref{eps_rect} does scan over $\Rcal_\epsilon$.

\subsubsection{Proof of Proposition \ref{prop:eps_time}}
First, the construction of $\textrm{dyad}$ takes $O(n^d)$ operations.
Indeed, the computation of $\textrm{dyad}_\ab(\tb)$ over $\ab \in [\log_2 n]^d \backslash \{1\}^d$ and $\tb \in [n/2^\ab]$ is done from Line~\ref{line:start_dyad_loop} to Line~\ref{line:end_dyad_loop} in \algref{eps_rect}, and is easily seen to require on the order of
\[
\sum_{\ab \in [\log_2 n]^d} \prod_j (n/2^{a_j}) \le n^d \Big(\sum_{a \ge 1} 2^{-a}\Big)^d = n^d
\] 
basic operations. 

Second, defining $a_+ = \sum_{j=1}^d a_j$, the convolution $\textrm{dyad}_\ab \ast b_{\fb}$ takes $O(n^d 2^{-a_+} \log n)$ operations with the FFT, since the convolution happens on a grid of size $\prod_j (n/2^{a_j}) = n^d 2^{-a_+}$.
Therefore, the computation on Line~\ref{line:shat} requires $O(n^d 2^{-a_+} \log n)$ basic operations.
Hence, once ${\rm dyad}$ is computed, computing $\hat \alpha$ requires on the order of 
\[
\sum_{\ab \in [\underline a, \overline a]^d} d \Big(\tprod_j |\Fcal_j|\Big) \Big(\frac{n^d}{2^{a_+}} \log n\Big) = O(\eps^{-2d} n^d \log n) \Big(\tsum_{a \ge \underline a} 2^{-a}\Big)^d = O(\eps^{-2d} n^d 2^{-d \underline a} \log n ),
\]
with $2^{-\underline a} = O(1/\epsilon^2 \underline h)$ since $\epsilon \underline h \ge 1$.
From this, we conclude.

\subsection*{Acknowledgments}

This work was partially supported by a grant from the US National Science Foundation (DMS 1223137).  
The authors would like to thank Zakhar Kabluchko for clarifying some technical points appearing in his work.

\bibliographystyle{chicago}
\bibliography{biblio}

\end{document}